\documentclass[12pt]{article}
\usepackage{amsmath}
\usepackage{amsfonts}
\usepackage{amssymb}
\usepackage{amsthm}
\usepackage[mathcal]{euscript}
\usepackage{color}
\usepackage{tikz-cd}

\newtheorem{theorem}{Theorem}
\newtheorem{lemma}{Lemma}
\newtheorem{proposition}{Proposition}
\newtheorem{corollary}{Corollary}
\theoremstyle{definition}
\newtheorem{definition}{Definition}
\theoremstyle{remark}
\newtheorem{remark}{Remark}
\newtheorem{conjecture}{Conjecture}

\def\id{\mathrm{id}}

\begin{document}

\title{On $q$-tensor product of Cuntz algebras}
\author{Alexey Kuzmin\footnote{Department of Mathematical Sciences, Chalmers University of Technology and University of Gothenburg, Gothenburg, Sweden, vagnard.k@gmail.com},
\and 
Vasyl Ostrovskyi\footnote{Institute of Mathematics. NAS of Ukraine, vo@imath.kiev.ua}, 
\and 
Danylo Proskurin\footnote{Faculty of Computer Sciences and Cybernetics, Kiev National Taras Shevchenko University, prosk@univ.kiev.ua} 
\and 
Roman Yakymiv\footnote{Faculty of Computer Sciences and Cybernetics, Kiev National Taras Shevchenko University, yakymiv@univ.kiev.ua}
}
\date{}

\maketitle
\begin{center}
{\it To 75-th birthday of our teacher Yurii S. Samoilenko}
\end{center}
\begin{abstract}
We consider $C^*$-algebra $\mathcal{E}_{n,m}^q$, which is a $q$-twist of two Cuntz-Toeplitz algebras. For the case $|q|<1$ we give an explicit formula, which untwists the $q$-deformation, thus showing that the isomorphism class of $\mathcal{E}_{n,m}^q$ does not depend on $q$. For the case $|q|=1$ we give an explicit description of all ideals in $\mathcal{E}_{n,m}^q$. In particular, $\mathcal{E}_{n,m}^q$ contains a unique largest ideal $\mathcal{M}_q$. Then we identify $\mathcal{E}_{n,m}^q / \mathcal{M}_q$ with the Rieffel deformation of  $\mathcal{O}_n \otimes \mathcal{O}_m$ and use a K-theoretical argument to show that the isomorphism class does not depend on $q$. 
\end{abstract}



\section{Introduction}
\label{intro}
\textbf{1.} Since early 80-th, a wide study of non-classical models of mathematical physics, quantum group theory and noncommutative probability (see e.g., \cite{BoSpe2,fiv,green,mac,MPe,zag}) gave rise to a number of papers on operator algebras generated by various deformed commutation relations, see \cite{BoSpe,Klimek,mar} etc. A general approach to the study of these relations has been provided by the framework of quadratic $*$-algebras allowing Wick ordering (Wick algebras), see \cite{jsw}. The class of Wick algebras includes, among others, deformations of canonical commutation relations of quantum mechanics, some quantum groups and quantum homogeneous spaces, see e.g., \cite{gisselson,Klimyk,vaksman1,woronowicz}.
On the other hand, one can consider Wick algebras as a deformation of Cuntz-Toeplitz algebra, see \cite{cun,dn,jsw}.

Let $\{T_{ij}^{kl},\ i,j,k,l=\overline{1,d}\}\subset\mathbb C$, $T_{ij}^{kl}=\overline{T}_{ji}^{lk}$. Wick algebra  $W(T)$, see \cite{jsw}, is the $*$-algebra generated by elements $a_j$, $a_j^*$, $j=\overline{1,d}$ subject to the relations
\[
a_i^*a_j=\delta_{ij}\mathbf 1+\sum_{k,l=1}^d T_{ij}^{kl} a_l a_k^*.
\]
It was shown in \cite{jsw} that properties of $W(T)$ depend on a self-adjoint operator $T$ called the operator of coefficients of $W(T)$. Namely, let $\mathcal{H}=\mathbb C^d$ and $e_1$, \dots, $e_d$ be the standard orthonormal basis of $\mathcal{H}$. Construct
\[
T\colon \mathcal{H}^{\otimes 2}\rightarrow \mathcal{H}^{\otimes 2},\quad
T e_k\otimes e_l=\sum_{i,j=1}^d T_{ik}^{lj} e_i\otimes e_j.
\]
Notice that the subalgebra of $W(T)$ generated by $\{a_j\}_{j=1}^d$ is free and can be identified with the full tensor algebra $\mathcal F=\bigoplus_{n=0}^{\infty} \mathcal{H}^{\otimes n}$ via
\[
a_{i_1} \dots a_{i_k} \mapsto e_{i_1} \otimes \dots \otimes e_{i_k} \in \mathcal{H}^{\otimes k}.
\]

\begin{definition}
The Fock representation of $W(T)$ is the unique irreducible $*$-representation $\pi_{F,T}$ determined by a cyclic vector $\Omega$, $||\Omega||=1$, such that
\[
\pi_{F,T} (a_j^*)\Omega =0,\ j=\overline{1,d}.
\]
\end{definition}

The problem of existence of $\pi_{F,T}$ is non-trivial and is one of the central problems in representation theory of Wick algebras. Some sufficient conditions are collected in the following theorem, see \cite{BoSpe,jps,jsw}.
\begin{theorem}\label{thm_fock}
The Fock representation $\pi_{F,T}$ of $W(T)$
exists if one of the conditions below is satisfied
\begin{itemize}
\item The operator of coefficients $T\ge 0$;
\item $||T||<\sqrt{2}-1$;
\item $T$ is \it{braided}, i.e. $(\mathbf 1\otimes T)(T\otimes\mathbf 1)
(\mathbf 1\otimes T)=(T\otimes\mathbf 1)
(\mathbf 1\otimes T)(\mathbf 1\otimes T)$ on $\mathcal{H}^{\otimes 3}$ and $||T||\le 1$. Moreover, if $||T||<1$ then $\pi_{F,T}$ is a faithful representation of $W(T)$ and $||\pi_{F,T}(a_j)||<(1-||T||)^{-\frac{1}{2}}$. If $||T||=1$, one can not guarantee boundedness of $\pi_{F,T}$ and in this case $\ker \pi_{F,T}$ is a $*$-ideal $\mathcal I_2$ generated as a $*$-ideal by $\ker (\mathbf 1+T)$. Hence $\pi_{F,T}$ is a faithful representation of $W(T)/\mathcal I_2$.
\end{itemize}
\end{theorem}

Another important question in the theory of Wick algebras is the question of stability of isomorphism classes of $\mathcal{W}(T) = C^*(W(T))$ for the case $||T|| < 1$. The following problem was posed in \cite{jsw2}.
\begin{conjecture}\label{q_stab}
Let $T : \mathcal{H}^{\otimes 2} \rightarrow \mathcal{H}^{\otimes 2}$ be a self-adjoint braided operator and $||T|| < 1$. Then $\mathcal{W}(T) \simeq \mathcal{W}(0)$.
\end{conjecture}

In particular, the authors of \cite{jsw2} have shown that the conjecture holds for the case $||T|| < \sqrt{2} - 1$, for more results on the subject see \cite{dn}, \cite{nk}.

Consider the case $T=0$ in a few more details. If $d=\dim \mathcal{H} =1$, then $W(0)$ is generated by a single isometry $s$, $s^*s=1$. In this case the universal $C^*$-algebra $\mathcal E$ of $W(0)$ exists and is isomorphic to the $C^*$-algebra generated by the unilateral shift $S$ in $l_2(\mathbb Z_+)$. Notice also that  $\pi_{F,0}(s)=S$, so the Fock representation of the $C^*$-algebra $\mathcal E$ is faithful.  The ideal $\mathcal I$ in $\mathcal E$, generated by $\mathbf 1 - ss^*$ is isomorphic to the algebra of compact operators and $\mathcal E/\mathcal I\simeq C(S^1)$, see \cite{coburn}. When $d\ge 2$, $W(0)$ is generated by $s_j$, $s_j^*$, such that
\[
s_i^*s_j=\delta_{ij}\mathbf 1,\quad i,j=\overline{1,d}.
\]
The Fock representation $\pi_{F,d}$ acts on $\mathcal F:=\mathcal{F}_d$ as follows
\begin{align*}
\pi_{F,d}(s_j)\Omega&=e_j,\quad \pi_{F,d}(s_j)e_{i_1}\otimes\cdots\otimes e_{i_k}=e_j\otimes e_{i_1}\otimes\cdots\otimes e_{i_k},\ k\ge 1,\\
\pi_{F,d}(s_j^*)\Omega&=0,\quad \pi_{F,d}(s_j^*)e_{i_1}\otimes\cdots\otimes e_{i_k}=\delta_{ji_1} e_{i_2}\otimes\cdots\otimes e_{i_k},\ k\ge 1.
\end{align*}
The universal $C^*$-algebra generated by $W(0)$ with $d\ge 2$ exists and is called the Cuntz-Toeplitz agebra $\mathcal O_d^{(0)}$. It is isomorphic to $C^*(\pi_{F,d}(W(0)))$, so the Fock representation of $\mathcal O_d^{(0)}$ is faithful, see \cite{cun}. Further, the ideal $\mathcal I$ generated by $1-\sum_{j=1}^d s_js_j^*$ is the unique largest ideal in $\mathcal O_d^{(0)}$. It is isomorphic to the algebra of compact operators on $\mathcal F_d$. The quotient $\mathcal O_d^{(0)}/\mathcal I$ is called the Cuntz algebra $\mathcal O_d$. It is nuclear (as well as $\mathcal O_d^{(0)}$),  simple and purely infinite, see \cite{cun} for more details.

\noindent
\textbf{2}. In this paper we study the $C^*$-algebras $\mathcal E_{n,m}^q$ generated by Wick algebras with operator of coefficients $T$ described as follows. Let $\mathcal{H}=\mathbb C^n\oplus\mathbb C^m$, $|q|\le 1$ and
\begin{align*}
& T u_1\otimes u_2  =0,\quad T v_1\otimes v_2=0,\quad u_1,u_2\in\mathbb C^n,\ v_1,v_2\in\mathbb C^m,\\
&T u\otimes v  = q v\otimes u,\quad T v\otimes u=\overline q u\otimes v,\quad u\in\mathbb C^n,\ v\in\mathbb C^m.
\end{align*}
We denote the corresponding Wick algebra by $WE_{n,m}^q$. Notice that $T$ satisfies the braid relation and $||T||=|q|\le 1$ for any $n,m\in\mathbb N$. In particular, the Fock representation $\pi_{F,q}$ exists for $|q|\le 1$ and is faithful on $WE_{n,m}^q$ for $|q|<1$.

The case $n=1$, $m=1$ was studied by various authors. Namely, $WE_{1,1}^q$ is generated by isometries $s_1$, $s_2$ subject to the relations \[s_1^*s_2= q s_2 s_1^*.\] It is easy to see that the corresponding universal $C^*$-algebra $\mathcal E_{1,1}^q$ exists for any $|q|\le 1$.

If $|q|<1$, the main result of \cite{jps2} states that  $\mathcal E_{1,1}^q\simeq\mathcal E_{1,1}^{(0)}=\mathcal O_2^{(0)}$ for any $|q|<1$. In particular the Fock representation of $\mathcal E_{1,1}^q$ is faithful.

The case $|q|=1$ was studied in \cite{kab,prolett,weber}. In this situation the additional relation \[s_2s_1=q s_1 s_2\]  holds in $\mathcal E_{1,1}^q$. It was shown that  $\mathcal E_{1,1}^q$ is nuclear for any $|q|=1$. Let $\mathcal M_q$ be the ideal generated by the projections $1-s_1s_1^*$ and $1-s_2s_2^*$. Then $\mathcal E_{1,1}^q/\mathcal M_q\simeq\mathcal A_q$, where $\mathcal A_q$ is the non-commutative torus, see \cite{rieff},
\[
\mathcal A_q=C^*(u_1,u_2\, |\, u_1^*u_1=u_1u_1^*=\mathbf 1,\
u_2^*u_2=u_2u_2^*=\mathbf 1,\ u_2^*u_1=q u_1 u_2^*).
\]
If $q$ is not a root of unity, then the corresponding non-commutative torus $\mathcal A_q$ is simple and $\mathcal M_q$ is the unique largest ideal in $\mathcal E_{1,1}^q$. Let us stress that unlike the case $|q|<1$, the $C^*$-isomorphism class of $\mathcal E_{1,1}^q$ is ``unstable'' with respect to $q$. Namely,
$\mathcal E_{1,1}^{q_1}\simeq \mathcal E_{1,1}^{q_2}$ iff $\mathcal A_{q_1}\simeq\mathcal A_{q_2}$, see \cite{kab,prolett,weber}.

One can consider another higher-dimensional analog of $\mathcal{E}_{1,1}^q$. For a set $\{q_{ij}\}_{i, j=1}^d$ of complex numbers such that $|q_{ij}|\le 1$, $q_{ij}=\overline q_{ji}$, $q_{ii}=1$, and $d>2$,  one can consider a $C^*$\nobreakdash-algebra $\mathcal E_{\{q_{ij}\}}$, generated by $s_j$, $s_j^*$, $j=\overline{1,d}$ subject to the relations
\[
s_j^*s_j=1,\quad s_i^*s_j=q_{ij}s_js_i^*.
\]

The case $|q_{ij}|<1$ was considered in \cite{kuzm_pochek}, where it was proved that $\mathcal{E}_{\{ q_{ij} \}}$ is nuclear and the Fock representation is faithful. It turned out that the fixed point $C^*$-subalgebra of $\mathcal E_{\{q_{ij}\}}$ with respect to the canonical action of $\mathbb T^d$ is an AF-algebra and is independent of $\{q_{ij}\}$. However the conjecture that
$\mathcal E_{\{q_{ij}\}}\simeq\mathcal E_{\{0\}}$ remains open.

The case $|q_{ij}|=1$  was studied in \cite{jeu_pinto,kab,prolett}. It was shown that $\mathcal E_{\{q_{ij}\}}$ is nuclear for any such family $\{q_{ij}\}$ and the Fock representation is faithful. For more details on ideal structure and representation theory see \cite{jeu_pinto,kab}.

We focus on the study of $\mathcal E_{n,m}^q$ with $n,m\ge 2$ (the case $n=1$, $m\ge 2$ will be considered separately, see \cite{yakym}). It is generated by isometries $\{s_j\}_{j=1}^n$,  and $\{t_r\}_{r=1}^m$, satisfying commutation relations of the following form
\begin{align}\label{baseqrel}
s_i^*s_j&=0, \quad 1\le i\ne j \le n, \notag
\\
 t_r^*t_s&=0,\quad 0\le r\ne s \le m,
 \\
 s_j^*t_r&=q t_r s_j^*, \quad 0\le j \le n, \ 0\le r\le m. \notag
\end{align}
The analysis is separated into two conceptually different cases, $|q|<1$ and $|q|=1$.

If $|q|<1$, we show that $\mathcal E_{n,m}^q\simeq\mathcal E_{n,m}^0= \mathcal O_{n+m}^{(0)}$, where the latter is the Cuntz-Toeplitz algebra with $n+m$ generators.

For the case $|q|=1$, we prove that $\mathcal E_{n,m}^q$ is nuclear, contains a unique largest ideal $\mathcal M_q$, and the quotient $\mathcal O_n\otimes_q\mathcal O_m:=\mathcal E_{n,m}^q/\mathcal M_q$ is simple and purely infinite for any $q$ specified above. Then we use the Kirchberg-Phillips classification Theorem, see \cite{Kirchberg,Philips}, to get one of our main results. Namely we show that
\[
\mathcal O_n\otimes_q\mathcal O_m\simeq\mathcal O_n\otimes\mathcal O_m
\]
for any $q\in\mathbb C$, $|q|=1$. Further we prove, that the Fock representation of $\mathcal E_{n,m}^q$ is faithful for any $|q|=1$ and use this fact to prove that $\mathcal E_{n,m}^q$ is isomorphic to the Rieffel deformation of $\mathcal O_n^{(0)}\otimes\mathcal O_m^{(0)}$. Next we show that the isomorphism class of $\mathcal M_q$ is independent on $q$ and consider $\mathcal E_{n,m}^q$ as an (essential) extension of $\mathcal O_n\otimes\mathcal O_m$ by $\mathcal M_q$ and study the corresponding $\mathsf{Ext}$ group. In particular, if $\gcd (n-1,m-1)=1$, this group is zero. Thus in this case, $\mathcal{E}_{n,m}^q$ and $\mathcal{E}_{n,m}^1$ both determine the zero class in $\mathsf{Ext}(\mathcal{O}_n \otimes_q \mathcal{O}_m, \mathcal{M}_q)$.
We stress that unlike the case of extensions by compacts, one can not immediately deduce that two trivial essential extensions are isomorphic. So the problem of isomorphism $\mathcal{E}_{n,m}^q \simeq \mathcal{E}_{n,m}^1$ remains open for further investigations.

\noindent
\textbf{3}. Recall how the algebras generated by isometries discussed above, are related to algebras of deformed canonical commutation relations.

We start with the case of one degree of freedom. The algebra $G_q$ of $q$-deformed canonical commutation relations, see \cite{bied,mac}, is generated by elements $a$, $a^*$  such that
\[
a^* a- q a a^*=\mathbf 1,
\]
where $q\in [-1,1]$. It is known, see \cite{jsw2}, that the universal $C^*$-algebra $\mathcal G_q$ generated by $G_q$ exists for $q\in [-1,1)$ and
$\mathcal G_q\simeq\mathcal E$ for any $q\in(-1,1)$.

The algebra $G_{q,d}$ of quon commutation relations with $d$ degrees of freedom was introduced and studied in \cite{BoSpe2,fiv,green,zag}. Namely, $G_{q,d}$ is generated by $a_j$, $a_j^*$, $j=\overline{1,d}$, subject to the commutation relations
\[
a_j^* a_i=\delta_{ij}\mathbf 1+ q a_i a_j^*,\quad i,j=\overline{1,d}, \quad q\in (0,1).
\]
Notice that the operator $T$, corresponding to $G_{q,d}$ has the form
\[
T e_{i}\otimes e_j= q e_{j}\otimes e_i
\]
so it is a braided contraction with $||T|| = q$. In particular, for $q<\sqrt{2}-1$ one has $\mathcal{G}_{q,d}\simeq \mathcal O_d^0$, where $\mathcal {G}_{q,d}$ is the $C^*$-algebra generated by $G_{q,d}$.

A multiparameter version of quons was considered in \cite{BoSpe,mar,MPe}. The corresponding $*$\nobreakdash-algebra $G_{\{q_{ij}\}}$, $q_{ij}=\overline{q}_{ji}$, $|q_{ij}|\le 1$, $i,j=\overline{1,d}$, is generated by
\[
a_i^* a_j=\delta_{ij}\mathbf 1 + q_{ij} a_j a_i^*,\quad i,j=\overline{1,d}.
\]
The operator $T$ acts as $T e_i\otimes e_j= q_{ij} e_j\otimes e_i$, so it is  a braided contraction as well. For $|q_{ij}|<\sqrt{2}-1$ we get
$\mathcal{G}_{\{q_{ij}\}}\simeq\mathcal O_d^0$. However, if $|q_{ij}|=1$ for all $i\ne j$, and $|q_{ii}| < 1$, then
${\mathcal G_{\{q_{ij}\}}\simeq\mathcal E_{\{q_{ij}\}} }$, see \cite{prolett}.

Take $k\in (0,1)$ and $q\in\mathbb C$, $|q|=1$. Construct $\mathcal H=\mathbb C^n\oplus\mathbb C^m$, $n$, $m\ge 2$ and define $T\colon\mathcal H^{\otimes 2}\rightarrow\mathcal H^{\otimes 2}$ as follows
\begin{align*}
T u_1\otimes u_2 &= k\, u_2\otimes u_1,\quad \mbox{if either}\ u_1,u_2\in\mathbb C^n\ \mbox{or}\ u_1,u_2\in\mathbb C^m\\
T u\otimes v &= q\, v\otimes u,\quad \mbox{if}\ u\in\mathbb C^n,\ v\in\mathbb C^m.
\end{align*}
Denote the corresponding Wick algebra by $WE_{n,m}^{q,k}$ and its universal $C^*$-algebra by $\mathcal E_{n,m}^{q,k}$. This $C^*$-algebra is generated by $s_j$, $t_r$, $j=\overline{1,n}$, $r=\overline{1,m}$, subject to the relations
\begin{align}\label{d_plectons}
s_i^*s_j&=\delta_{ij}\mathbf 1 + k\, s_j s_i^*,\nonumber\\
t_r^* t_l&=\delta_{rl}\mathbf 1+ k\, t_l t_r^*,\\
s_j^* t_r &= q\, t_r s_j^*,\ t_r s_j= q\, s_j t_r\nonumber.
\end{align}
Relations (\ref{d_plectons}) can be regarded as an example of system considered in \cite{BoLyt} in the case of finite degrees of freedom. Applying the general stability result, we get that $\mathcal E_{n,m}^{q,k}\simeq\mathcal E_{n,m}^q$ for $k<\sqrt{2}-1$.

Notice that for $k=\pm 1$ we get a discrete analogue of commutation relations for generalized statistics introduced in \cite{LiM}.

\section{The case $|q|<1$}
We start with some lemmas. Let $\Lambda_n$ denote the set of all words in alphabet $\{\overline{1,n}\}$. For any non-empty $\mu=(\mu_1,\ldots,\mu_k)$, and a family of elements $b_1$, \dots, $b_n$, we denote by $b_{\mu}$ the product $b_{\mu_1}\cdots b_{\mu_k}$; we also put $b_{\emptyset}=\mathbf{1}$. In this section we assume that any word $\mu$ belongs to $\Lambda_n$.
\begin{lemma}
Let $Q=\sum_{i=1}^n s_i s_i^*$, then
\[
\sum_{|\mu|=k} s_{\mu}Q s_{\mu}^* =\sum_{|\nu|=k+1} s_{\nu}s_{\nu}^*.
\]
\end{lemma}
\begin{proof}
Straightforward.
\end{proof}

\begin{lemma}
For any $x\in\mathcal{E}_{n,m}^q$ one has
\[
\Bigl\|\sum_{|\mu|=k}s_{\mu}x s_{\mu}^*\Bigr\| \le \|x\|.
\]
\end{lemma}
\begin{proof}
 $1$. First prove the claim for positive $x$. In this case one has $0\le x\le \|x\|\mathbf{1}$. Hence $0\le s_{\mu} x s_{\mu}^* \le \|x\| s_{\mu}s_{\mu}^*$, and
\[
\Bigl\| \sum_{|\mu|=k}s_{\mu}x s_{\mu}^*\Bigr\|\le \|x\|\cdot \Bigl\|\sum_{|\mu|=k} s_{\mu}s_{\mu}^*\Bigr\|.
\]
Note that $s_{\mu}^* s_{\lambda} =\delta_{\mu\,\lambda}$, $\mu,\lambda\in\Lambda_n$, $|\mu|=|\lambda|=k$, implying that $\{s_{\mu}s_{\mu}^*\ |\ |\mu|=k\}$ form a family of pairwise orthogonal projections. Hence $\|\sum_{|\mu|=k} s_{\mu}s_{\mu}^*\|=1$, and the statement for positive $x$ is proved.

 $2$. For any $x\in\mathcal{E}_{n,m}^q$. write $A=\sum_{|\mu|=k}s_{\mu} x s_{\mu}^*$, then $A^*=\sum_{|\mu|=k}s_{\mu} x s_{\mu}^*$ and
\[
A^* A=\sum_{|\mu|=k}s_{\mu} x^* x s_{\mu}^*.
\]
Then by the proved above,
\[
\|A\|^2 =\|A^*A\|\le \|x^*x\|=\|x\|^2. \qedhere
\]
\end{proof}

Construct $\widetilde{t}_l=(\mathbf{1}-Q) t_l$,  $l=\overline{1,m}$.

\begin{lemma}
The following commutation relations hold
\begin{align*}
&s_i^*\widetilde{t}_l=0,\quad i=\overline{1,n},\quad l=\overline{1,m},
\\
&\widetilde{t}_r^*\widetilde{t}_l=0,\quad l\ne r\quad l,r=\overline{1,m},
\\
& \widetilde{t}_r^*\widetilde{t}_r=\mathbf{1}-|q|^2 Q>0,\quad r=\overline{1,m}.
\end{align*}
\end{lemma}

\begin{proof} We have $s_i^* (\mathbf{1}-Q)=0$, implying that $s_i^*\widetilde{t}_l=0$ for any $i=\overline{1,n}$, and $l=\overline{1,m}$.

Further,
\begin{align*}
\widetilde{t}_r^*\widetilde{t}_l&=t_r^* (\mathbf{1}-Q)t_l=t_r^* t_l-
\sum_{i=1}^n t_r^* s_i s_i^* t_l=\delta_{rl}\mathbf{1}-\sum_{i=1}^n |q|^2 s_i t_r^* t_l s_i^*=\\
&=\delta_{rl}(\mathbf{1}-|q|^2 Q).\qedhere
\end{align*}
\end{proof}

\begin{proposition}\label{gener1}
For any $r=\overline{1,m}$, one has
\[
t_r=\sum_{k=0}^{\infty}\sum_{|\mu|=k} q^k s_{\mu}\widetilde{t}_r s_{\mu}^*.
\]
In particular, the family $\{s_i,\widetilde{t}_r,\ i=\overline{1,n},\ r=\overline{1,m}\}$ generates $\mathcal{E}_{n,m}^q$.
\end{proposition}

\begin{proof}
Put $M_k^{r}=\sum_{|\mu|=k} q^k s_{\mu}\widetilde{t}_r s_{\mu}^*$, $k\in\mathbb{Z}_{+}$. Then
\[
M_0^r=t_r- Q t_r=t_r-\sum_{|\mu|=1} s_{\mu}s_{\mu}^* t_r,
\]
and
\begin{align*}
M_k^r&=\sum_{|\mu|=k} q^k s_{\mu} (\mathbf{1}-Q) t_r s_{\mu}^*=
\sum_{|\mu|=k}  s_{\mu} (\mathbf{1}-Q) s_{\mu}^*t_r =\\
&=\sum_{|\mu|=k}  s_{\mu} s_{\mu}^* t_r - \sum_{|\mu|=k+1} s_{\mu} s_{\mu}^* t_r.
\end{align*}
Then
\[
S_N^r=\sum_{k=0}^{N} M_k^r=t_r-\sum_{|\mu|=N+1}s_{\mu}s_{\mu}^* t_r=
t_r- q^{N+1}\sum_{|\mu|=N+1}s_{\mu}t_rs_{\mu}^*.
\]
Since $\|\sum_{|\mu|=N+1}s_{\mu}t_rs_{\mu}^*\|\le \|t_r\|=1$ one has that $S_N^r\rightarrow t_r$ in $\mathcal{E}_{n,m}^q$ as $N\to\infty$.
\end{proof}

Suppose that $\mathcal{E}_{n,m}^q$ is realised by Hilbert space operators. Consider the left polar decomposition $\widetilde{t}_r =\widehat{t}_r\cdot c_r$, where $c_r^2=\widetilde{t}_r^*\widetilde{t}_r=\mathbf{1}-|q|^2 Q>0$, implying that $\widehat{t}_r$ is an isometry and
\[
\widehat{t}_r=\widetilde{t}_r c_r^{-1}\in\mathcal{E}_{n,m}^q,\quad r=\overline{1,m}.
\]

\begin{lemma}
The following commutation relations hold
\begin{align*}
s_i^*\widehat{t}_r&=0,\quad i=\overline{1,n},\ r=\overline{1,m},
\\ \widehat{t}_r^*\widehat{t}_l&=\delta_{rl}\mathbf{1},\quad r,l=\overline{1,m}.
\end{align*}
\end{lemma}

\begin{proof}
Indeed, for any $i=\overline{1,n}$, and $r=\overline{1,m}$. one has
\[
s_i^* \widehat{t}_r=s_i^*\widetilde{t}_r\, c_r^{-1}=0,
\]
and
\[
\widehat{t}_r^*\widehat{t}_l=c_r^{-1}\widetilde{t}_r^*\widetilde{t}_l c_r^{-1}=0,\quad r\ne l.
\qedhere
\]
\end{proof}

Summing up the results stated above, we get the following

\begin{theorem}\label{zer_q_gen}
Let  $\widehat{t}_r=(\mathbf{1}-Q)t_r (\mathbf{1}-|q|^2 Q)^{-\frac{1}{2}}$, $r=\overline{1,m}$. Then the family $\{s_i,\widehat{t}_r\}_{i=1}^{n}{}_{r={1}}^{m}$ generates $\mathcal{E}_{n,m}^q$, and
\[
s_i^*s_j=\delta_{ij}\mathbf{1},\quad t_r^* t_l=\delta_{rl}\mathbf{1},\quad s_i^* t_r =0,\qquad i,j=\overline{1,n},\ r,l=\overline{1,m}.
\]
\end{theorem}

\begin{proof}
It remains to note that $\widetilde{t}_r=\widehat{t}_r (1-|q|^2 Q)^{\frac{1}{2}}$, so $\widetilde{t}_r\in C^* (\widehat{t}_r,\, Q)$, so by Proposition \ref{gener1} the elements $s_i$, $\widehat{t}_r$, $i=\overline{1,n}$, $r=\overline{1,m}$, generate $\mathcal{E}_{n,m}^q$.
\end{proof}

\begin{corollary}
Denote by $v_i$, $i=\overline{1,n}+m$, the isometries generating $\mathcal{E}_{n,m}^0=\mathcal{O}_{n+m}^0$. Then Theorem \ref{zer_q_gen} implies that the correspondence
\[
v_i\mapsto s_i,\ i=\overline{1,n},\quad v_{n+r}\mapsto \widehat{t}_r,\ r=\overline{1,m},
\]
extends uniquely to a surjective homomorphism $\varphi\colon\mathcal{E}_{n,m}^0\rightarrow\mathcal{E}_{n,m}^q$.
\end{corollary}

Our next aim is to construct the inverse homomorphism $\psi\colon\mathcal{E}_{n,m}^q\rightarrow\mathcal{E}_{n,m}^0$. To do it, put
\[
\widetilde{Q}=\sum_{i=1}^n v_i v_i^*\quad \widetilde{w}_r=v_{n+r}(1-|q|^2\widetilde{Q})^{\frac{1}{2}}, \qquad r=\overline{1,m}.
\]
 Then $\widetilde{w}_r^*\widetilde{w}_r=1-|q|^2\widetilde{Q}$, and $\widetilde{w}_r^*\widetilde{w}_l=0$ if $r\ne l$, $r,l=\overline{1,m}$. Construct
\[
w_r=\sum_{k=0}^{\infty}\sum_{|\mu|=k}q^k v_{\mu}\widetilde{w}_r v_{\mu}^*,\quad r=\overline{1,m},
\]
where $\mu$ runs over $\Lambda_n$, and set as above $v_{\mu}=v_{\mu_1}\cdots v_{\mu_k}$. Note that the series above converges with respect to norm in $\mathcal{E}_{n,m}^0$.

\begin{lemma}
The following commutation relations hold
\[
w_r^*w_l=\delta_{rl}\mathbf{1},\quad v_i^* w_r=q w_r v_i^*,\quad i=\overline{1,n},\ r,l=\overline{1,m}.
\]
\end{lemma}

\begin{proof}
First we note that $v_i^*\widetilde{w}_r=0$, and $\widetilde{w}_r^* v_i=0$ for any $i=\overline{1,n}$, and $j=\overline{1,m}$, implying that
\[
v_{\delta}^*\widetilde{w}_r=0,\quad \widetilde{w}_r^* v_{\delta}=0,\quad \mbox{for any nonempty}\ \delta\in\Lambda_n,\ r=\overline{1,m}.
\]
Let $|\lambda|\ne |\mu|$, $\lambda,\mu\in\Lambda_n$. If $|\lambda|>|\mu|$, then $\lambda=\widehat{\lambda}\gamma$ with $|\lambda|=|\mu|$ and
\[
v_{\lambda}^*v_{\mu}=\delta_{\widehat{\lambda}\mu}v_{\gamma}^*.
\]
Otherwise $\mu=\widehat{\mu}\beta$, $|\widehat{\mu}|=|\lambda|$ and
\[
v_{\lambda}^*v_{\mu}=\delta_{\lambda\widehat{\mu}}v_{\beta}.
\]
So, if $|\lambda|>|\mu|$ one has
\[
v_{\lambda}\widetilde{w}_r^* v_{\lambda}^* v_{\mu}\widetilde{w}_r v_{\mu}^*=
\delta_{\widehat{\lambda}\mu} v_{\lambda}\widetilde{w}_r^* v_{\gamma}^*\widetilde{w}_r v_{\mu}=0,
\]
and if $|\mu|>|\lambda|$, then
\[
v_{\lambda}\widetilde{w}_r^* v_{\lambda}^* v_{\mu}\widetilde{w}_r v_{\mu}^*=
\delta_{\lambda\widehat{\mu}} v_{\lambda}\widetilde{w}_r^* v_{\beta}\widetilde{w}_r v_{\mu}=0.
\]
Since $v_{\mu}^*v_{\lambda}=\delta_{\mu\lambda}\mathbf{1}$, if $|\mu|=|\lambda|$, one has
\begin{align*}
 w_r^*w_r&=\lim_{N\rightarrow\infty}\Bigl(\sum_{k=0}^N \sum_{|\lambda|=k}|q|^{k} v_{\lambda}\widetilde{w}_r^* v_{\lambda}^*\Bigr)\cdot
\Bigl(\sum_{l=0}^N \sum_{|\mu|=l}|q|^{l} v_{\mu}\widetilde{w}_r v_{\mu}^*\Bigr)
\\
&=\lim_{N\rightarrow\infty}\sum_{k,l=0}^N \sum_{|\lambda|=k,|\mu|=l}|q|^{k+l} v_{\lambda}\widetilde{w}_r^* v_{\lambda}^* v_{\mu}\widetilde{w}_r v_{\mu}^*=
\lim_{N\rightarrow\infty}\sum_{k=0}^N \sum_{|\lambda|,|\mu|=k,}|q|^{2k} v_{\lambda}\widetilde{w}_r^* v_{\lambda}^* v_{\mu}\widetilde{w}_r v_{\mu}^*
\\
&=\lim_{N\rightarrow\infty}\sum_{k=0}^N \sum_{|\mu|=k}|q|^{2k} v_{\mu}\widetilde{w}_r^*\widetilde{w}_r v_{\mu}^*=
\lim_{N\rightarrow\infty}\sum_{k=0}^N \sum_{|\mu|=k}|q|^{2k} v_{\mu}(\mathbf{1}-|q|^2\widetilde{Q}^2) v_{\mu}^*
\\
&=\lim_{N\rightarrow\infty}\sum_{k=0}^N \Bigl (\sum_{|\mu|=k}|q|^{2k} v_{\mu}v_{\mu}^*-\sum_{|\mu|=k+1} |q|^{2k+2}v_{\mu} v_{\mu}^* \Bigr)
\\
&=\lim_{N\rightarrow\infty}\Bigl(\mathbf{1}-|q|^{2N+2}\sum_{|\mu|=N+1} v_{\mu}v_{\mu}^*\Bigr)=\mathbf{1}.
\end{align*}
Since $\widetilde{w}_r^*\widetilde{w}_l=0$, $r\ne l$, the same arguments as above imply that $w_r^* w_l=0$, $r\ne l$.

For any non-empty $\mu\in\Lambda_n$ write $\sigma(\mu)=\emptyset$ if $|\mu|=1$, and
$\sigma(\mu)=(\mu_2,\ldots,\mu_k)$ if ${|\mu|=k>1}$. Further, for any $i=\overline{1,n}$, $r=\overline{1,m}$ one has
\begin{align*}
v_i^* w_r&=\sum_{k=0}^{\infty}\sum_{|\mu|=k} q^k s_i^* v_{\mu}\widetilde{w}_r v_{\mu}^*=
v_i^*\widetilde{w}_r+\sum_{k=1}^{\infty}\sum_{|\mu|=k} q^k \delta_{i\mu_1} v_{\sigma(\mu)}\widetilde{w}_r v_{\sigma(\mu)}^*v_i^*
\\
&=q\sum_{k=0}^{\infty}\sum_{|\mu|=k} q^k v_{\mu}\widetilde{w}_r v_{\mu}^*v_i^*=q w_r v_i^*.
\qedhere
\end{align*}
\end{proof}

\begin{lemma}
For any $r=\overline{1,m}$, one has $\widetilde{w}_r=(\mathbf{1}-\widetilde{Q})w_r$.
\end{lemma}

\begin{proof}
First note that $(\mathbf{1}-\widetilde{Q})v_i=0$, $i=\overline{1,n}$, implies that
\[
(\mathbf{1}-\widetilde{Q})v_{\mu}=0,\quad |\mu|\in\Lambda_n,\ \mu\ne\emptyset.
\]
Then
\begin{align*}
(\mathbf{1}-\widetilde{Q})w_r&=(\mathbf{1}-\widetilde{Q})\Bigl(\sum_{k=0}\sum_{|\mu|=k} q^k v_{\mu}\widetilde{w}_r v_{\mu}^*\Bigr)
\\
&=(\mathbf{1}-\widetilde{Q})\widetilde{w}_r+\sum_{k=1}\sum_{|\mu|=k} q^k (\mathbf{1}-\widetilde{Q})v_{\mu}\widetilde{w}_r v_{\mu}^*=(\mathbf{1}-\widetilde{Q})\widetilde{w}_r.
\end{align*}

To complete the proof it remains to note that $\widetilde{Q}v_{n+r}=0$, $r=\overline{1,m}$. So,
\[
\widetilde{Q}\widetilde{w}_r=\widetilde{Q}v_{n+r}(\mathbf{1}-|q|^2\widetilde{Q})^{\frac{1}{2}}=0.
\]
\end{proof}

\begin{theorem}\label{q_ser_gen}
Let $v_i$, $i=\overline{1,n}+m$, be the isometries generating $\mathcal{E}_{n,m}^0$, and $\widetilde{Q}=\sum_{i=1}^n v_i v_i^*$. Put
\[
\widetilde{w}_r=v_{n+r}(\mathbf{1}-|q|^2\widetilde{Q})^{\frac{1}{2}}\quad \mbox{and}\quad w_r=\sum_{k=0}\sum_{|\mu|=k} q^k v_{\mu}\widetilde{w}_r v_{\mu}^*.
\]
Then
\[
v_i^*v_j=\delta_{ij}\mathbf{1},\quad w_r^* w_l=\delta_{rl}\mathbf{1},\quad v_i^* w_r=q w_r v_i^*,\quad i,j=\overline{1,n},\ r,l=\overline{1,m}.
\]
Moreover, the family $\{v_i,w_r\}_{i=1}^{n}{}_{r=1}^{m}$ generates $\mathcal{E}_{n,m}^0$.
\end{theorem}

\begin{proof}
We need to prove only the last statement of the theorem.  We have
\[
v_{n+r}=\widetilde{w}_r (\mathbf{1}-|q|^2\widetilde{Q})^{-\frac{1}{2}}=(\mathbf{1}-\widetilde{Q})
w_r(\mathbf{1}-|q|^2\widetilde{Q})^{-\frac{1}{2}}\in C^*(w_r, v_i,\ i=\overline{1,n}).
\]
Hence $v_i$, $w_r$, $i=\overline{1,n}$, $r=\overline{1,m}$, generate $\mathcal{E}_{n,m}^0$.
\end{proof}

\begin{corollary}\label{cor2}
The statement of Theorem \ref{q_ser_gen} and the universal property of $\mathcal{E}_{n,m}^q$ imply the existence of a surjective homomorphism $\psi\colon\mathcal{E}_{n,m}^q\rightarrow\mathcal{E}_{n,m}^0$ defined by
\[
\psi(s_i)=v_i,\quad \psi(t_r)=w_r,\quad i=\overline{1,n},\ r=\overline{1,m}.
\]
\end{corollary}

Now we are ready to formulate the main result of this section.

\begin{theorem}\label{qstab_thm}
For any $q\in\mathbb{C}$, $|q|<1$, one has an isomorphism $\mathcal{E}_{n,m}^q\simeq\mathcal{E}_{n,m}^0$.
\end{theorem}

\begin{proof}
In Theorem \ref{zer_q_gen}, we constructed the surjective homomorphism $\varphi\colon\mathcal{E}_{n,m}^0\rightarrow\mathcal{E}_{n,m}^q$ defined by
\[
\varphi(v_i)=s_i,\quad \varphi(v_{n+r})=\widehat{t}_r,\quad i=\overline{1,n},\ r=\overline{1,m}.
\]
Show that $\psi\colon\mathcal{E}_{n,m}^q\rightarrow\mathcal{E}_{n,m}^0$ from Corollary \ref{cor2} is the inverse of $\varphi$. Indeed, the equalities  ${\psi(s_i)=v_i}$, $i=\overline{1,n}$, imply that
\[
\psi (\mathbf{1}-Q)=\mathbf{1}-\widetilde{Q}.
\]
Then, since $\psi(t_r)=w_r$, we get
\[
\psi (\widetilde{t}_r)=\psi ((\mathbf{1}-Q)t_r)=(\mathbf{1}-\widetilde{Q})w_r=\widetilde{w}_r,\quad r=\overline{1,m},
\]
and
\[
\psi(\widehat{t}_r)=\psi(\widetilde{t}_r (\mathbf{1}-|q|^2 Q)^{-\frac{1}{2}})=
\widetilde{w}_r (\mathbf{1}-|q|^2 \widetilde{Q})^{-\frac{1}{2}}=v_{n+r},\quad r=\overline{1,m}.
\]
So, $\psi\varphi(v_i)=\psi (s_i)=v_i$, $\psi\varphi(v_{n+r})=\psi(\widehat{t}_r)=v_{n+r}$, $i=\overline{1,n}$, $r=\overline{1,m}$, and
\[
\psi\varphi= id_{\mathcal{E}_{n,m}^0}.
\]

Show that $\varphi\psi=id_{\mathcal{E}_{n,m}^q}$. Indeed,
\[
\varphi(\widetilde{w}_r)=\varphi(v_{n+r}(\mathbf{1}-|q|^2\widetilde{Q})^{\frac{1}{2}})=
\widehat{t}_r (\mathbf{1}-|q|^2 Q)^{\frac{1}{2}}=\widetilde{t}_r,\quad r=\overline{1,d}.
\]
Then for any $r=\overline{1,m}$, one has
\[
\varphi(w_r)=\sum_{k=0}\sum_{|\mu|=k}q^k\varphi(v_{\mu})\varphi(\widetilde{w}_r)\varphi{v_{\mu}^*}=
\sum_{k=0}\sum_{|\mu|=k} q^ks_{\mu} \widetilde{t}_r s_{\mu}^*= t_r.
\]
So, $\varphi\psi(s_i)=\varphi(v_i)=s_i$, $\varphi\psi(t_r)=\varphi(w_r)=t_r$, $i=\overline{1,n}$, $r=\overline{1,m}$.
\end{proof}

\section {The case $|q|=1$}
In this section, we discuss the case $|q|=1$. Notice that for $|q|=1$, the relations in $\mathcal E_{n,m}^q$
imply that $t_js_i = q s_it_j$, $i=\overline{1,n}$, $j=\overline{1, m}$. Indeed, for $B_{ij} = t_js_i - qs_it_j$ we
have directly $B_{ij}^* B_{ij} =0$.

\subsection{Auxiliary results}
In this subsection we collect some general facts about $C^*$-dynamical systems, crossed products and Rieffel deformations which we will use in our considerations.

\subsubsection{Fixed point subalgebras}\label{group_action}
First we recall how properties of a fixed point subalgebra of a $C^*$-algebra with an action of a compact group are related to properties of the whole algebra.

\begin{definition}
Let $A$ be a $C^*$-algebra with an action $\gamma$ of a compact group $G$. A fixed point subalgebra $A^\gamma$ is a subset of all $a \in A$ such that $\gamma_g(a) = a$ for all $g \in G$.
\end{definition}

Notice  that for every action of a compact group $G$ on a $C^*$-algebra $A$ one can construct a faithful conditional expectation $E_{\gamma} : A \rightarrow A^\gamma$ onto the fixed point subalgebra, given by
\[ E_{\gamma}(a) = \int_{G} \gamma_g(a) d \lambda, \]
where $\lambda$ is the Haar measure on $G$.

 A homomorphism $\varphi : A \rightarrow B$ between $C^*$-algebras with actions $\alpha$ and $\beta$ of a compact group $G$ is called equivariant if
 \[ \varphi \circ \alpha_g = \beta_g \circ \varphi \text{ for any } g \in G.
 \]
\begin{proposition}[\cite{Ozawa}, Section 4.5, Theorem 1, 2]\label{faith_fix}

\begin{enumerate}
    \item Let $\gamma$ be an action of a compact group $G$ on a $C^*$-algebra $A$. Then $A$ is nuclear if and only if $A^\gamma$ is nuclear.
    \item Let $\varphi : A \rightarrow B$ be an equivariant $*$-homomorphism. Then $\varphi$ is injective on $A$ if and only if $\varphi$ is injective on $A^\alpha$.
\end{enumerate}
\end{proposition}

\subsubsection{Crossed products}

Given a locally compact group $G$ and a $C^*$-algebra $A$ with a $G$-action $\alpha$, consider the full crossed product $C^*$-algebra $A \rtimes_\alpha G$, see \cite{williams}. One has two natural embeddings into the multiplier algebra  $M(A \rtimes_\alpha G)$,
\begin{gather*}
i_A : A \rightarrow M(A \rtimes_\alpha G), \ i_G : G \rightarrow M(A \rtimes_\alpha G), 
\\
(i_A(a) f)(s) = a f(s), \quad (i_G(t)f)(s) = \alpha_t(f(t^{-1} s)), \quad t,s\in G, \ a \in A,
\end{gather*}
for $f \in C_c(G, A)$.
\begin{remark}\label{rem2}
Obviously, $i_G(s)$ is a unitary element of $M(A\rtimes_{\alpha} G)$ for any $s\in G$. Recall that $i_G$ determines the following homomorphism denoted also by $i_G$
\[
i_G\colon C^*(G)\rightarrow M(A\rtimes_{\alpha} G)
\]
defined by
\[
i_G (f)=\int_G f(s) i_G(s) d\lambda(s),
\]
where $\lambda$ is the left Haar measure on $G$.

Notice that for any $g\in C_c(G,A)$ one has
\[
(i_G(f) g)(t)=f\cdot_{\alpha} g,
\]
where $\cdot_\alpha$ denotes the product in $A\rtimes_{\alpha} G$. In particular, when $A$ is unital we can identify $i_G(f)$ with $f\cdot_\alpha \mathbf{1}_A$, and in fact $i_G$ maps $C^*(G)$ into $A \rtimes_\alpha G$. Also notice that
\[
i_G(t) i_A (a) i_G(t)^{-1}  =i_A( \alpha _t(a)) \in M(A \rtimes_\alpha G).
\]
\end{remark}

If  $\varphi$ is an equivariant homomorphism between $C^*$-algebras $A$ with a $G$-action $\alpha$ and $B$ with a $G$-action $\beta$, then one can define the homomorphism
\[
\varphi \rtimes G : A \rtimes_\alpha G \rightarrow B \rtimes_\beta G, \ (\varphi \rtimes G)(f)(t) = \varphi(f(t)),\quad f\in C_c(G,A).
\]
Let $A$ be a unital $C^*$-algebra with $G$-action $\alpha$. Then $\iota_A\colon\mathbb{C}\rightarrow A$,
\[
\iota_A(\lambda)=\lambda\mathbf 1_A,
\]
is an equivariant homomorphism, where $G$ acts trivially on $\mathbb{C}$. Since $\mathbb{C}\rtimes G=C^* (G)$, one has that
\[
\iota_A\rtimes G\colon C^*(G)\rightarrow A\rtimes_{\alpha} G.
\]
In fact, in this case we have
\begin{equation}\label{unital_embedding}
    \iota_A \rtimes G = i_G,
\end{equation}
where $i_G\colon C^*(G)\rightarrow A\rtimes_{\alpha} G$ is described in Remark \ref{rem2}. Indeed, for any $g\in G_c (G,A)$ one has
\begin{equation*}
    \begin{split}
        (i_G(f)\cdot_\alpha g)(s) & = \int_G f(t) \alpha_t(g(t^{-1}s)) dt
        \\
        & = \int_G f(t) 1_A \alpha_t(g(t^{-1}s)) dt
        \\
        & = ((f(\cdot)1_A) \cdot_\alpha g)(s)=((\iota_A \rtimes G)(f)\cdot_{\alpha}g)(s),
    \end{split}
\end{equation*}
implying $i_G(f)=(\iota_A\rtimes G)(f)$ for any $f\in C^*(G)$.

\subsubsection{Rieffel's deformation}

Below, we recall some basic facts on Rieffel's deformations. Given a $C^*$-algebra $A$ equipped with an action $\alpha$ of $\mathbb{R}^n$ and a skew symmetric matrix
$\Theta \in M_n(\mathbb R)$, one can construct the Rieffel deformation of $A$, denoted by $A_\Theta$, see \cite{anderson,rieffel_primar}. In particular the elements $a \in A$ such that $x \mapsto \alpha_x(a) \in C^\infty(\mathbb{R}^n, A)$ form a dense subset $A_\infty$ in $A_\Theta$ and for any $a,b\in  A_\infty$ their product in $A_\Theta$ is given by the following oscillatory integral (see \cite{rieffel_primar}):
\begin{equation}\label{rieffel_product}
    a \cdot_\Theta b :=
\int_{\mathbb{R}^n} \int_{\mathbb{R}^n} \alpha_{\Theta(x)}(a) \alpha_y(b) e^{2 \pi i \langle x, y \rangle} dx dy,
\end{equation}
where $\langle \cdot,\cdot\rangle$ is a scalar product in $\mathbb R^n$.

In what follows, we will be interested in periodic actions of $\mathbb{R}^n$, i.e., we assume that $\alpha$ is an action of $\mathbb{T}^n$.
Given a character $\chi \in \widehat{\mathbb{T}}^n \simeq \mathbb{Z}^n$, consider
\[
A_\chi = \{ a \in A : \alpha_z(a) = \chi(z)a \text{ for every } z \in \mathbb{T}^n \} .
\]
Then
\[
A=\overline{\bigoplus_{\chi\in\mathbb Z^n} A_{\chi}},
\]
where some terms could be equal to zero, and $A_{\chi_1}\cdot A_{\chi_2}\subset A_{\chi_1+\chi_2}$, $A_{\chi}^*=A_{-\chi}$. So,  $A_\chi$, $\chi \in \mathbb Z^n$, can be treated as homogeneous components of
$\mathbb Z^n$-grading on $A$. Conversely, any $\mathbb{Z}^n$-grading of $A$ determines an action of $\mathbb{T}^n$ on $A$:
for $a \in A_p$ we let $\alpha_t(a) = e^{2 \pi i \langle t, p \rangle} a$ (see, e.g., \cite{williams}).

For a periodic action $\alpha$ of $\mathbb{R}^n$ on a $C^*$-algebra $A$ and a skew-symmetric matrix $\Theta \in M_n(\mathbb{R})$, construct the Rieffel deformation $A_\Theta$. Notice that all
homogeneous elements belong to $A_{\infty}$. Apply formula (\ref{rieffel_product}) to $a \in A_p$, $b \in A_q$:
\begin{equation*}
    \begin{split}
        a \cdot_\Theta b & =  \int_{\mathbb{R}^n}
        \int_{\mathbb{R}^n} e^{2 \pi i \langle \Theta(x), p \rangle} a e^{2 \pi i \langle y, q \rangle}
        b e^{2 \pi i \langle x, y \rangle} dx\, dy
        \\
        & = a \cdot b \int_{\mathbb{R}^n} e^{2 \pi i \langle y, q \rangle}
        \int_{\mathbb{R}^n} e^{2 \pi i \langle x, -\Theta(p) \rangle} e^{2 \pi i \langle x, y \rangle} dx\, dy
        \\
        & = a \cdot b \int_{\mathbb{R}^n} e^{2 \pi i \langle y, q \rangle} \delta_{y - \Theta(p)} \,dy
        \\
        & = e^{2 \pi i \langle \Theta(p), q \rangle} a \cdot b .
    \end{split}
\end{equation*}
Thus, given $a \in A_p$ and $b \in A_q$ one has
\begin{equation} \label{homogeneous_rieff_product}
    a \cdot_\Theta b = e^{2 \pi i \langle \Theta(p), q \rangle}a \cdot b .
\end{equation}

\begin{remark}
Notice that $A_\Theta$ also possesses a $\mathbb{Z}^n$-grading such that $(A_\Theta)_p = A_p$ for every $p \in \mathbb{Z}^n$. Due to \eqref{homogeneous_rieff_product}, we have $a\cdot_\Theta b = a\cdot b$ for any $a,b\in A_{\pm p}$, $p\in\mathbb{Z}^n$. Indeed, for any skew
symmetric $\Theta\in M_n(\mathbb R^n)$ and $p\in\mathbb Z^n$, one has $\langle\Theta\, p\, ,\, \pm p\rangle=0$. The involution on $(A_\Theta)_p$ coincides with the involution on $A_p$.
\end{remark}

Consider a $C^*$-dynamical system $(A,\mathbb T^n, \alpha)$, and its covariant
representation $(\pi,U)$ on a Hilbert space $\mathcal H$. For any $p\in \mathbb{Z}^n\simeq  \mathbb{\widehat{T}}^n$, put
\[
 \mathcal H_p = \{ h\in \mathcal H \mid U_t h = e^{2\pi i \langle t,p\rangle} h\}.
\]
Then $\mathcal H = \bigoplus _{p\in \mathbb Z^n} \mathcal H_p$ (see \cite{williams}).

\begin{proposition}[\cite{warped}, Theorem 2.8]\label{theta_rep}
Let $(\pi,U)$ be a covariant representation of $(A,\mathbb T^n,\alpha)$ on a Hilbert space $\mathcal{H}$. Then one can define a representation $\pi_\Theta$ of $A_\Theta$ as follows:
\[
\pi_\Theta(a) \xi = e^{2 \pi i \langle \Theta(p), q \rangle} \pi(a) \xi,
\]
for every $\xi \in \mathcal{H}_q$, $a \in A_p$, $p,q\in \mathbb Z^n$.
Moreover, $\pi_\Theta$ is faithful if and only if $\pi$ is faithful.
\end{proposition}

 It is known that Rieffel's deformation can be embedded into $M(A \rtimes_\alpha \mathbb{R}^n)$, but for the periodic actions we have an explicit description of this embedding.

 \begin{proposition}[\cite{Amandip}, Lemma 3.1.1]
The following mapping defines an embedding
\[ i_{A_\Theta} : A_\Theta \rightarrow M(A \rtimes_\alpha \mathbb R^n), \ i_{A_\Theta}(a_p) = i_A(a_p)i_{\mathbb R^n}(-\Theta(p)), \]
where $p \in \mathbb{Z}^n$ and $a_p$ is homogeneous of degree $p$.
\end{proposition}

\begin{proposition}[\cite{Kasprzak_rieffel}, Proposition 3.2 and \cite{Amandip}, Section 3.1]\label{prop4}
Let $(A,\mathbb R^n,\alpha)$ be a $C^*$\nobreakdash-dyna\-mical system with  periodic $\alpha$ and unital $A$.  Put $A_{\Theta}$ to be the Rieffel deformation of $A$. There exist a periodic
action $\alpha^{\Theta}$ of $\mathbb R^n$ on $A_{\Theta}$ and an isomorphism $\Psi : A_\Theta \rtimes_{\alpha^\Theta} \mathbb{R}^n \rightarrow A \rtimes_\alpha \mathbb{R}^n$ such that the following diagram is commutative
\[ \begin{tikzcd}
& C^*(\mathbb{R}^n) \simeq C_0(\mathbb{R}^n) \arrow[dl, "i_{\mathbb{R}^n}"'] \arrow[dr, "i_{\mathbb{R}^n}"] & \\
A_\Theta \rtimes_{\alpha^\Theta} \mathbb{R}^n \arrow[rr, "\Psi"] & & A \rtimes_\alpha \mathbb{R}^n
\end{tikzcd} \]
\end{proposition}

Namely, $\alpha^{\Theta}(a)=\alpha(a)$ holds for any $a\in A_p$, $p\in\mathbb{Z}^n$. Then it is easy to verify that $i_{A_{\Theta}}\colon A_{\Theta}\rightarrow M(A\rtimes_{\alpha}\mathbb{R}^n)$ with $i_{\mathbb R^n}\colon\mathbb R^n\rightarrow M(A\rtimes_\alpha \mathbb R^n)$ determine a covariant representation of $(A_{\Theta},\mathbb{R}^n,\alpha^{\Theta})$ in $M(A\rtimes_{\alpha}\mathbb R^n)$. Hence, by the universal property of crossed product we get the corresponding homomorphism
\[
\Psi\colon A_{\Theta}\rtimes_{\alpha^\Theta}\mathbb{R}^n\rightarrow M(A\rtimes_\alpha \mathbb R^n).
\]
In fact, the range of $\Psi$ coincides with $A\rtimes_\alpha \mathbb R^n$ and $\Psi$ defines an isomorphism
\begin{equation}\label{PsiHom}
    \Psi\colon A_{\Theta}\rtimes_{\alpha^\Theta}\mathbb{R}^n\rightarrow A\rtimes_\alpha \mathbb R^n,
\end{equation}
see \cite{Kasprzak_rieffel,Amandip} for more detailed considerations.

The following propositions shows that Rieffel's deformation inherits properties of the non-deformed counterpart.
\begin{proposition}[\cite{Kasprzak_rieffel}, Theorem 3.10]\label{Rieff_nuclear}
A $C^*$-algebra $A_\Theta$ is nuclear if and only if $A$ is nuclear.
\end{proposition}

\begin{proposition}[\cite{Kasprzak_rieffel}, Theorem 3.13]\label{Rieff_K_theory}
For a $C^*$-algebra $A$ one has
\[
K_0(A_\Theta) = K_0(A)\quad \mbox{and}\quad K_1(A_{\Theta})=K_1(A).
\]
\end{proposition}

\subsubsection{Rieffel's deformation of a tensor product }
In this part we apply Rieffel's deformation procedure to a tensor product of two nuclear unital $C^*$-algebras equipped with an action of $\mathbb{T}$.

Let $A$, $B$ be $C^*$-algebras with actions $\alpha$ and $\beta$ of $\mathbb{T}$. Then there is a natural action $\alpha \otimes \beta$ of $\mathbb{T}^2$ on $A \otimes B$ defined as
\[ (\alpha \otimes \beta)_{\varphi_1, \varphi_2}(a \otimes b) = \alpha_{\varphi_1}(a) \otimes \beta_{\varphi_2}(b). \]
Consider the induced gradings on $A$ and $B$:
\[ A = \bigoplus_{p_1 \in \mathbb{Z}} A_{p_1}, \quad B = \bigoplus_{p_2 \in \mathbb{Z}} B_{p_2}. \]
Then the corresponding grading on $A \otimes B$ is
\[ A \otimes B := \bigoplus_{(p_1, p_2)^t \in \mathbb{Z}^2} A_{p_1} \otimes B_{p_2}. \]
In particular, $a \otimes \mathbf 1 \in (A \otimes B)_{(p_1, 0)^t}$ and $\mathbf 1 \otimes b \in (A \otimes B)_{(0, p_2)^t}$, where $a \in A_{p_1}$ and $b \in B_{p_2}$.

Given $q = e^{2 \pi i \varphi_0}$, consider
\begin{equation}\label{theta_q}
\Theta_q = \left( \begin{array}{cc}
    0 & \frac{\varphi_0}{2}  \\
    -\frac{\varphi_0}{2} & 0
\end{array} \right).
\end{equation}
We construct the Rieffel deformation $(A \otimes B)_{\Theta_q}$.

\begin{proposition}\label{Rieffel_tensor_product}
One has the following homomorphisms
\begin{align*}
\eta_A \colon A &\rightarrow (A \otimes B)_{\Theta_q}, \ \eta_A(a) = a \otimes \mathbf 1,
\\
\eta_B \colon B &\rightarrow (A \otimes B)_{\Theta_q}, \ \eta_B(b) = \mathbf 1 \otimes b,
\end{align*}
such that for homogeneous elements $a \in A_{p_1}$ and $b \in B_{p_2}$ it holds
 \[
 \eta_B(b) \cdot_{\Theta_q} \eta_A(a) = e^{2 \pi i p_1 p_2 \varphi_0}  \eta_A(a) \cdot_{\Theta_q} \eta_B(b).
 \]
\end{proposition}

\begin{proof}
Recall that
$\mathbb Z^2$-homogeneous
components of $A \otimes B$ and $(A \otimes B)_{\Theta_q}$ coincide and will be considered the same. Let $e_1=(1,0)^t$, $e_2=(0,1)^t$.

Given $a \in A_p$, we have
\[
 \eta_A(a) = a \otimes \mathbf 1 \in ((A \otimes B)_{\Theta_q})_{p\, e_1},
\]
implying that
\[
\eta_A(a)^* = a^* \otimes \mathbf 1 \in ((A \otimes B)_{\Theta_q})_{-p\, e_1}.
\]
Let $a_1 \in A_{p_1}$ and $a_2 \in A_{p_2}$. Then
\[
\eta_A(a_1) \cdot_{\Theta_q} \eta_A(a_2) = e^{2 \pi i \langle p_1 \Theta_q(e_1), p_2 e_1 \rangle} (a_1 \otimes \mathbf 1)(a_2 \otimes \mathbf1) = a_1 a_2 \otimes \mathbf1 = \eta_A(a_1 a_2).
\]
Thus $\eta_A$ is a homomorphism. The arguments for $\eta_B$ are the same.

Given $a \in A_{p_1}$ and $b \in B_{p_2}$, one has
\begin{align*}
\eta_A(a) \cdot_{\Theta_q} \eta_B(b) &= e^{ 2 \pi i \langle \Theta_q(p_1 e_1), p_2 e_2 \rangle } (a \otimes \mathbf1) (\mathbf1 \otimes b) = e^{- \pi i p_1 p_2 \varphi_0 } a \otimes b,
\\
\eta_B(b) \cdot_{\Theta_q} \eta_A(a) &= e^{ 2 \pi i \langle \Theta_q(p_2 e_2), p_1 e_1 \rangle } (\mathbf1 \otimes b)(a \otimes \mathbf1)  = e^{ \pi i p_1 p_2 \varphi_0 } a \otimes b,
\end{align*}
implying that
\[
\eta_B(b) \cdot_{\Theta_q} \eta_A(a) = e^{2 \pi i p_1 p_2 \varphi_0}  \eta_A(a) \cdot_{\Theta_q} \eta_B(b) . \]
\end{proof}

\subsection{Fock representation of $\mathcal E_{n,m}^q$}
In this part we show that the Fock representation of $\mathcal E_{n,m}^q$ is faithful, and apply this result to show that $\mathcal E_{n,m}^q$ is isomorphic to the Rieffel deformation $(\mathcal O_n^{(0)}\otimes\mathcal O_m^{(0)})_{\Theta_q}$, where $\Theta_q$ is specified in  (\ref{theta_q}).
\begin{definition}
The Fock representation of $\mathcal E_{n,m}^q$ is the unique up to unitary equivalence irreducible $*$-representation $\pi_F^q$ determined by the action  on vacuum vector $\Omega$, $||\Omega||=1$,
\[
\pi_F^q(s_j^*)\Omega = 0,\quad \pi_F^q(t_r^*)\Omega =0,\quad j=\overline{1,n},\ r=\overline{1,m}.
\]
\end{definition}

Denote by $\pi_{F,n}$ the Fock representation of $\mathcal{O}_n^{(0)}\subset\mathcal E_{n,m}^q$ acting on the space
\[
\mathcal{F}_n=\mathcal{T}(\mathcal{H}_n)= \mathbb{C}\Omega\oplus\bigoplus_{d=1}^{\infty} \mathcal{H}_n^{\otimes d},\quad \mathcal{H}_n=\mathbb{C}^n,
\]
described by formulas
\begin{align*}
\pi_{F,n}(s_j)\Omega& =e_j,\quad \pi_{F,n}(s_j)e_{i_1}\otimes e_{i_2}\cdots\otimes e_{i_d}=e_j\otimes e_{i_1}\otimes e_{i_2}\cdots\otimes e_{i_d},\\
\pi_{F,n}(s_j^*)\Omega& =0,\quad \pi_{F,n}(s_j^*)e_{i_1}\otimes e_{i_2}\otimes\cdots\otimes e_{i_d}=\delta_{ji_1}e_{i_2}\otimes\cdots\otimes e_{i_d},\quad d\in\mathbb{N},
\end{align*}
where $e_1$, \dots, $e_n$ is the standard orthonormal basis of $\mathcal{H}_n$. Notice that $\pi_{F,n}$ is the unique irreducible faithful representations of $\mathcal{O}_n^{(0)}$, see for example \cite{jsw2}.

Recall that the Fock representation of $\mathcal E_{n,m}^q$ exists for any $q\in\mathbb{C}$, $|q|\le 1$.
For $|q|=1$,  one has $\|T\|=1$, and the kernel of the Fock representation of the Wick algebra $WE_{n,m}^q$ coincides with the $*$-ideal $\mathcal I_2$ generated by $\ker (\mathbf 1+T)$, see Introduction. In our case,
\[
\mathcal{I}_2=\langle t_r s_j-q s_j t_r,\ j=\overline{1,n},\ r=\overline{1,n}\rangle.
\]
Denote by $E_{n,m}^q$ the quotient $WE_{n,m}^q/\mathcal I_2$. Obviously, $\mathcal E_{n,m}^q=C^*(E_{n,m}^q)$. So one has the following corollary of Theorem \ref{thm_fock}.

\begin{proposition}\label{dense_faithful}
The Fock representation of $\mathcal E_{n,m}^q$ exists and is faithful on the $*$-subalgebra $E_{n,m}^q\subset\mathcal E_{n,m}^q$.
\end{proposition}

Below we give an explicit formula for $\pi_F (s_j)$,
$\pi_F (t_r)$. Consider the Fock representations $\pi_{F,n}$ and $\pi_{F,m}$ of $*$-subalgebras $C^*(\{s_1, \ldots, s_n\})=\mathcal O_n^{(0)}\subset\mathcal E_{n,m}^q$ and $C^*(\{t_1, \ldots, t_m\})=\mathcal{O}_m^{(0)}\subset \mathcal E_{n,m}^q$ respectively. Denote by $\Omega_n\in\mathcal F_n$ and $\Omega_m\in\mathcal F_m$ the corresponding vacuum vectors.
\begin{theorem}
The Fock representation $\pi_F^q$ of $\mathcal E_{n,m}^q$ acts on the space $\mathcal{F}=\mathcal{F}_n\otimes\mathcal{F}_m$ as follows
\begin{align*}
\pi_F^q(s_j)&=\pi_{F,n} (s_j)\otimes d_m(q^{-\frac{1}{2}}),\quad j = \overline{1,n},\\
\pi_F^q(t_r)&= d_n(q^{\frac{1}{2}})\otimes \pi_{F,m} (t_r),\quad r=\overline{1,m},
\end{align*}
where $d_k(\lambda)$ acts on $\mathcal F_k$, $k=n,m$ by
\[
d_k(\lambda)\Omega_k =\Omega_k,\quad d_k(\lambda)X=\lambda^l X,\quad X\in \mathcal{H}_k^{\otimes l},\quad l\in\mathbb{N}.
\]
\end{theorem}
\begin{proof}
It is a direct calculation to verify that the operators defined above satisfy the relations of $\mathcal E_{n,m}^q$. Since $\pi_{F,k}$ is irreducible on $\mathcal{F}_k$, $k=m,n$, the representation $\pi_F^q$ is irreducible on $\mathcal{F}_n\otimes\mathcal{F}_m$. Finally put $\Omega=\Omega_n\otimes\Omega_m$, then obviously
\[
\pi_F^q(s_j^*)\Omega=0,\ \mbox{and}\ \pi_F^q(t_r^*)\Omega=0,\quad j=\overline{1,n}, r=\overline{1,m}
\]
Thus $\pi_F^q$ is the Fock representation of $\mathcal E_{n,m}^q$.
\end{proof}
\begin{remark}\label{rem_fock}
In some cases, it will be more convenient to present the operators of the Fock representation of $\mathcal E_{n,m}^q$ in one of the alternative forms,
\begin{align*}
\pi_F^q(s_j)&=\pi_{F,n} (s_j)\otimes\mathbf{1}_{\mathcal F_m},\quad j = \overline{1,n},
\\
\pi_F^q(t_r)&= d_n(q)\otimes \pi_{F,m} (t_r),\quad r=\overline{1,m},
\end{align*}
or
\begin{align*}
\pi_F^q(s_j)&=\pi_{F,n} (s_j)\otimes d_m(q^{-1}),\quad j = \overline{1,n},\\
\pi_F^q(t_r)&= \mathbf{1}_{\mathcal F_n}\otimes \pi_{F,m} (t_r),\quad r=\overline{1,m},
\end{align*}
which are obviously unitary equivalent to the one presented in the statement above.
\end{remark}

Consider the action $\alpha$ of $\mathbb{T}^2$ on $\mathcal{E}_{n,m}^q$,
\[
\alpha_{\varphi_1, \varphi_2}(s_i) = e^{2\pi i \varphi_1} s_i, \quad \alpha_{\varphi_1, \varphi_2}(t_r) = e^{2 \pi i \varphi_2} t_r.
\]
Recall, see Section \ref{group_action}, that the conditional expectation, associated to $\alpha$ is denoted by $E_\alpha$.

\begin{proposition}\label{enmq_fixpoint}
The fixed point $C^*$-subalgebra $(\mathcal E_{n,m}^q)^\alpha$ is an AF-algebra.
\end{proposition}

\begin{proof}
The family $\{s_{\mu_1}s_{\nu_1}^*t_{\mu_2}t_{\nu_2}^*,\ \mu_1,\nu_1\in\Lambda_n,\ \mu_2,\nu_2\in\Lambda_m\}$ is dense in $\mathcal E_{n,m}^q$, thus the family
$\{E_{\alpha}(s_{\mu_1}s_{\nu_1}^*t_{\mu_2}t_{\nu_2}^*),\ \mu_1,\nu_1\in\Lambda_n,\ \mu_2,\nu_2\in\Lambda_m\}$ is dense in $(\mathcal E_{n,m}^q)^\alpha$. Further,
\[
E_{\alpha}(s_{\mu_1}s_{\nu_1}^*t_{\mu_2}t_{\nu_2}^*)=0,\ \mbox{if}\
|\mu_1|\ne |\nu_1|\ \mbox{or}\ |\mu_2|\ne |\nu_2|,
\]
and $E_{\alpha}(s_{\mu_1}s_{\nu_1}^*t_{\mu_2}t_{\nu_2}^*)=s_{\mu_1}s_{\nu_1}^*t_{\mu_2}t_{\nu_2}^*$ otherwise. Hence
\[
(\mathcal E_{n,m}^q)^\alpha=c.l.s.\{s_{\mu_1}s_{\nu_1}^*t_{\mu_2}t_{\nu_2}^*,\ |\mu_1|=|\nu_1|, |\mu_2|=|\nu_2|,\ \mu_1,\nu_1\in\Lambda_n,\
\mu_2,\nu_2\in\Lambda_m\}.
\]
Put $\mathcal A_{1,0}^0=\mathbb C$,
\[
\mathcal A_{1,0}^{k_1}=c.l.s.\{s_{\mu_1}s_{\nu_1}^*, |\mu_1|=|\nu_1|=k_1,\ \mu_1,\nu_1\in\Lambda_n\},\quad k_1\in\mathbb N,
\]
and $\mathcal A_{2,0}^0=\mathbb C$,
\[
\mathcal A_{2,0}^{k_2}=c.l.s.\{t_{\mu_2}t_{\nu_2}^*, |\mu_2|=|\nu_2|=k_2,\ \mu_1,\nu_1\in\Lambda_n\},\quad k_2\in\mathbb N.
\]
It is easy to see that $xy=yx$, $x\in\mathcal A_{1,0}^{k_1}$, $y\in\mathcal A_{2,0}^{k_2}$. Let
\[
\mathcal A_{k}^\alpha=\sum_{k_1+k_2=k}\mathcal A_{1,0}^{k_1}\cdot\mathcal A_{2,0}^{k_2}.
\]
Evidently $\mathcal A_k^\alpha$ is a finite-dimensional subalgebra in $(\mathcal E_{n,m}^q)^\alpha$ for any $k\in\mathbb Z_{+}$ and
\[
(\mathcal E_{n,m}^q)^{\alpha}=\overline{\bigcup_{k\in\mathbb Z_{+}}\mathcal A_{k}^\alpha}.\qedhere
\]
\end{proof}

\begin{remark}
Define unitary operators $U_{\varphi_1,\varphi_2}$,  $(\varphi_1,\varphi_2)\in\mathbb{T}^2$ on $\mathcal{F}_n \otimes \mathcal{F}_m$ as follows:
\[
U_{\varphi_1, \varphi_2}= d_n(e^{2\pi i\varphi_1}) \otimes d_m(e^{2\pi i\varphi_2}).
\]
Then $(\pi_F^q, U_{\varphi_1,\varphi_2})$ is a covariant representation of ($\mathcal{E}_{n,m}^q$, $\mathbb{T}^2$, $\alpha$).
\end{remark}

\begin{theorem}
The Fock representation $\pi_F^q$ of $\mathcal E_{m,n}^q$ is faithful.
\end{theorem}

\begin{proof}
Consider the action $\alpha^\pi$ of $\mathbb T^2$ on $\pi_F^q(\mathcal E_{n,m}^q)$ induced by the action $\alpha$ on $\mathcal{E}_{n,m}^q$
\begin{align*}
\alpha_{\varphi_1,\varphi_2}^\pi \pi_F^q(s_j)&=e^{2\pi i\varphi_1}\pi_F^q(s_j):=S_{j,\varphi_1,\varphi_2},
\\
\alpha_{\varphi_1,\varphi_2}^\pi \pi_F^q(t_r)&=e^{2\pi i\varphi_2}\pi_F^q(t_r):=T_{r,\varphi_1,\varphi_2}.
\end{align*}
To see that $\alpha^q_{\varphi_1,\varphi_2}$ is an automorphism of $\pi_F^q (\mathcal E_{n,m}^q)$ we notice that the operators $S_{j,\varphi_1,\varphi_2}$, $T_{r,\varphi_1,\varphi_2}$, satisfy the defining relations in $\mathcal E_{n,m}^q$, and
\[
S_{j,\varphi_1,\varphi_2}^*\Omega=
T_{r,\varphi_1,\varphi_2}^*\Omega=0.
\]
Evidently the family $\{S_{j,\varphi_1,\varphi_2},S_{j,\varphi_1,\varphi_2}^*,
T_{r,\varphi_1,\varphi_2},
T_{r,\varphi_1,\varphi_2}^*\}_{j=1}^n{}_{r=1}^m$ is irreducible and therefore defines the Fock representation of $\mathcal{E}_{n,m}^q$. Thus, by the uniqueness of the Fock representation, there exists a unitary $V_{\varphi_1,\varphi_2}$ on $\mathcal F = \mathcal{F}_n \otimes \mathcal{F}_m$ such that for any $j=\overline{1,n}$, $r=\overline{1,m}$, one has
\[
S_{j,\varphi_1,\varphi_2}=\mathsf{Ad}(V_{\varphi_1,\varphi_2})\circ \pi_F^q (S_j),\quad
T_{r,\varphi_1,\varphi_2}=\mathsf{Ad}(V_{\varphi_1,\varphi_2})\circ \pi_F^q (T_r),
\]
implying that $\alpha^\pi_{\varphi_1,\varphi_2}$ is an automorphism of $\pi_F^q (\mathcal E_{n,m}^q)$ for any  $(\varphi_1,\varphi_2)\in\mathbb T^2$. Evidently,
\[
\pi_F^q\colon\mathcal E_{n,m}^q\rightarrow \pi_F^q (\mathcal E_{n,m}^q)
\]
is equivariant with respect to $\alpha$ and $\alpha^\pi$.

By Proposition \ref{faith_fix}, the representation $\pi_F^q$ is faithful on $\mathcal{E}_{n,m}^q$ if and only if it is faithful on $(\mathcal E_{n,m}^q)^\alpha$. Further, by Proposition \ref{enmq_fixpoint},
\[
(\mathcal E_{n,m}^q)^{\alpha}=\overline{\bigcup_{k\in\mathbb Z_{+}}\mathcal A_{k}^\alpha}.
\]
Evidently $\mathcal A_{k}^\alpha\subset E_{n,m}^q$, $k\in\mathbb Z_{+}$. Hence by Proposition \ref{dense_faithful}, $\pi_{F}^q$ is faithful on $\mathcal A_k^{\alpha}$ for any $k\in\mathbb Z_{+}$. It is an easy exercise to show that a representation of an AF-algebra is injective if and only if it is injective on the finite-dimensional subalgebras.
\end{proof}

The next step is to construct a representation of $(\mathcal O_n^{(0)}\otimes\mathcal O_m^{(0)})_{\Theta_q}$ corresponding to the Fock representation $\pi_{F,n}\otimes \pi_{F,m}$ of $\mathcal O_n^{(0)}\otimes\mathcal O_m^{(0)}$.

The pair $(\pi_{F,n} \otimes \pi_{F,m},\ U_{\varphi_1,\varphi_2})$ determines a covariant representation of $(\mathcal{O}_n^0 \otimes \mathcal{O}_m^0,\mathbb{T}^2,\alpha)$, where as above
\[
\alpha_{\varphi_1,\varphi_2}(s_j\otimes\mathbf 1)=e^{2\pi i\varphi_1} (s_j\otimes\mathbf 1),\quad \alpha_{\varphi_1,\varphi_2}(\mathbf 1\otimes t_r)=
e^{2\pi i\varphi_2} (\mathbf 1\otimes t_r).
\]
Notice that for $p=(p_1,p_2)^t\in\mathbb{Z}_{+}^2$, the subspace $\mathcal{H}_n^{\otimes p_1}\otimes \mathcal{H}_m^{\otimes p_2}$ is the $(p_1,p_2)^t$-homogeneous component of $\mathcal F$ related to the action of $U_{\varphi_1,\varphi_2}$, and $(\mathcal F)_p=\{0\}$ for any $p\in\mathbb Z^2\setminus\mathbb Z_{+}^2 $.

Recall also that $\widehat{s}_j=s_j\otimes \mathbf{1}$ is contained in $e_1=(1,0)^t$-homogeneous component and $\widehat{t}_r=\mathbf 1\otimes t_r$ is in $e_2=(0,1)^t$-homogeneous component with respect to $\alpha$. Now one can apply Proposition \ref{theta_rep}. Namely, given $\xi=\xi_1\otimes\xi_2 \in \mathcal{H}_n^{\otimes p_1} \otimes \mathcal{H}_m^{\otimes p_2}$ one gets
\begin{align*}
(\pi_{F,n} \otimes \pi_{F,m})_{\Theta_q}(\widehat{s}_j)\, \xi &= e^{2 \pi i \langle\Theta_q\, e_1,\, p  \rangle }\, \pi_{F,n} \otimes \pi_{F,m}(\widehat{s}_j)\, \xi = \\
&= \pi_{F,n}(s_j)\xi_1 \otimes e^{-\pi i\, p_2\, \varphi_0}\, \xi_2=(\pi_{F,n}(s_j)\otimes d_m(q^{-\frac{1}{2}}))\, \xi,
\end{align*}
and
\begin{align*}
(\pi_{F,n} \otimes \pi_{F,m})_{\Theta_q}(\widehat{t}_r)\, \xi &= e^{2 \pi i \langle\Theta_q\, e_2,\, p  \rangle }\, \pi_{F,n} \otimes \pi_{F,m}(\widehat{t}_r)\xi = \\
&=e^{\pi i\, p_1\, \varphi_0}\, \xi_1 \otimes
\pi_{F,m}(t_r)\, \xi_2= ( d_n(q^{\frac{1}{2}})\otimes \pi_{F,m}(t_r))\,\xi.
\end{align*}
Notice that for any $j=\overline{1,n}$, and $r=\overline{1,m}$,
\[
(\pi_{F,n}\otimes \pi_{F,m})_{\Theta_q}(\widehat{s}_j^*)\Omega=0,\quad
(\pi_{F,n}\otimes \pi_{F,m})_{\Theta_q}(\widehat{t}_r^*)\Omega=0.
\]

\begin{theorem}\label{cuntoep_rieff}
For any $q\in\mathbb{C}$, $|q|=1$, the $C^*$-algebra $\mathcal E_{n,m}^q$ is isomorphic to $(\mathcal{O}_n^{(0)}\otimes\mathcal{O}_m^{(0)})_{\Theta_q}$.
\end{theorem}

\begin{proof}
Proposition \ref{Rieffel_tensor_product} implies that elements
$(\mathcal O_n\otimes\mathcal O_m)_{\Theta_q}\ni\widehat{s}_j=s_j\otimes\mathbf 1$ and
$(\mathcal O_n\otimes\mathcal O_m)_{\Theta_q}\ni\widehat{t}_r=\mathbf 1\otimes t_r$ satisfy
\[
\widehat{s}_j^*\widehat{s}_i=\delta_{ij}\mathbf 1\otimes\mathbf 1,\quad
\widehat{t}_r^*\widehat{t}_s=\delta_{rs}\mathbf 1\otimes\mathbf 1,\quad
\widehat{t}_r^*\widehat{s}_j= q \widehat{s}_j\widehat{t}_r^*.
\]
Hence, by the universal property one can construct a surjective homomorphism $\Phi\colon \mathcal E_{n,m}^q \rightarrow (\mathcal{O}_n^{(0)}\otimes\mathcal{O}_m^{(0)})_{\Theta_q}$ defined by
\[
\Phi (s_j)=\widehat{s}_j,\quad \Phi(t_r)=\widehat{t}_r,\quad j=\overline{1,n},\ r=\overline{1,m}.
\]
Notice that due to the considerations above, $\pi_F^q=(\pi_{F,n}\otimes \pi_{F,m})_{\Theta_q}\circ\Phi$. Since $\pi_F^q$ is faithful representation of $\mathcal E_{n,m}^q$, we deduce that $\Phi$ is injective.
\end{proof}

The nuclearity of $\mathcal O_n^{(0)}\otimes\mathcal O_m^{(0)}$ and Proposition \ref{Rieff_nuclear} immediately imply the following

\begin{corollary}
 The $C^*$-algebra $\mathcal{E}_{n,m}^q$ is nuclear for any $q\in\mathbb C$, $|q|=1$.
\end{corollary}

The nuclearity of $\mathcal E_{n,m}^q$ can also be shown using  more explicit arguments. One can use the standard trick of untwisting the $q$-deformation in the crossed product, which clarifies  informally the nature of isomorphism (\ref{PsiHom}). Namely, for $q = e^{2 \pi i \varphi_0}$
consider the action $\alpha_q$ of $\mathbb{Z}$ on $\mathcal E_{n,m}^q$ defined on the generators as
\[
\alpha_q^k(s_j)=e^{\pi i k\varphi_0} s_j,\quad \alpha_q^k(t_r)=e^{-\pi i k\varphi_0} t_r,\quad
j=1,\dots,n,\ r=1,\dots,m,\ k\in\mathbb Z.
\]
Denote by the same symbol the similar action on
$\mathcal E_{n,m}^1\simeq \mathcal{O}_n^{(0)}\otimes\mathcal{O}_m^{(0)}$.
Here we denote by $\widetilde{s}_j$ and $\widetilde{t}_r$ the generators of $\mathcal E_{n,m}^1$.

\begin{proposition}
 For any  $\varphi_0\in [0,1)$, one has an isomorphism
 $\mathcal E_{n,m}^q\rtimes_{\alpha_q} \mathbb Z\simeq \mathcal E_{n,m}^1\rtimes_{\alpha_q} \mathbb Z$.
\end{proposition}
\begin{proof}
 Recall that $\mathcal E_{n,m}^1\rtimes_{\alpha_q} \mathbb Z$ is generated as a $C^*$-algebra by elements
 $\widetilde{s}_j$, $\widetilde{t}_r$ and a unitary $u$, such that the following relations satisfied
 \[
   u \widetilde{s}_ju^*=e^{i\pi \varphi_0}\, \widetilde{s}_j,\quad
    u \widetilde{t}_ru^*=e^{-i\pi \varphi_0}\, \widetilde{t}_r,\quad j=\overline{1,n},\ r=\overline{1,m}.
 \]
Put $\widehat{s}_j=\widetilde{s}_j\, u$ and $\widehat{t}_r=\widetilde{t}_r\, u$. Obviously,
$\widehat{s}_j$, $\widehat{t}_r$ and $u$ generate $\mathcal{E}_{n,m}^1\rtimes_{\alpha_q} \mathbb Z$. Further,
\[
 \widehat{s}_j^*\widehat{s}_k=\delta_{jk}\mathbf 1,\quad \widehat{t}_r^*\widehat{t}_l=\delta_{rl}\mathbf{1}
\]
and
\[
\widehat{s}_j\widehat{t}_r=\widetilde{s}_j u\widetilde{t}_r u=e^{-i\pi\varphi_0}\widetilde{s}_j\widetilde{t}_r u^2=
e^{-i\pi\varphi_0}\widetilde{t}_r\widetilde{s}_j u^2=
e^{-2\pi i\varphi_0}\widetilde{t}_r u\widetilde{s}_j u=\overline{q}\, \widehat{s}_j\widehat{t}_r.
\]
In a similar way we get $\widehat{s}_j^*\widehat{t}_r= q\widehat{t}_r\widehat{s}_j^*$, $j=\overline{1,n}$, $r=\overline{1,m}$.
Finally
\[
u\widehat{s}_j u^*=e^{i\pi\varphi_0}\widehat{s}_j,\quad u\widehat{t_r}u^*=e^{-i\pi\varphi_0}\widehat{t}_r.
\]
Hence the correspondence
\[
 s_j\mapsto \widehat{s}_j,\quad t_j\mapsto \widehat{t}_j,\quad u\mapsto u,
\]
determines a homomorphism
$\Phi_q\colon \mathcal E_{n,m}^q\rtimes_{\alpha_q} \mathbb Z \rightarrow \mathcal E_{n,m}^1\rtimes_{\alpha_q} \mathbb Z$.
The inverse is constructed evidently.
\end{proof}

Let us show the nuclearity of $\mathcal{E}_{n,m}^q$ again. Indeed, $\mathcal{E}_{n,m}^1=\mathcal O_{n}^{(0)}\otimes\mathcal{O}_m^{(0)}$ is nuclear. Then so is the crossed product
 $\mathcal E_{n,m}^1\rtimes_{\alpha_q} \mathbb Z$. Then due to the isomorphism above,
 $\mathcal E_{n,m}^q\rtimes_{\alpha_q} \mathbb Z$ is nuclear, implying the nuclearity of $\mathcal E_{n,m}^q$, see
 \cite{Black}.

We finish this part by an analog of the well-known Wold decomposition theorem for a single isometry. Recall that
\[ Q = \sum_{j = 1}^n s_j s_j^*, \ P = \sum_{r = 1}^m t_r t_r^*. \]
\begin{theorem}[Generalised Wold decomposition]\label{wold_dec}
Let $\pi\colon\mathcal E_{n,m}^q\rightarrow \mathbb{B}(\mathcal H)$ be a $*$-repre\-sen\-ta\-tion. Then
\[
\mathcal H=\mathcal H_1\oplus\mathcal H_2\oplus\mathcal H_3\oplus\mathcal H_4,
\]
where each $\mathcal H_j$, $j=1,2,3,4$, is invariant with respect to $\pi$, and for
$\pi_j=\pi\restriction_{\mathcal H_j}$ one has
\begin{itemize}
\item $\mathcal H_1=\mathcal F\otimes\mathcal K$ for some Hilbert space $\mathcal{K}$, and $\pi_1=\pi_F^q\otimes\mathbf{1}_{\mathcal K}$;
\item $\pi_2(\mathbf{1}-Q)=0$, $\pi_2(\mathbf 1 -P)\ne 0$;
\item  $\pi_3(\mathbf{1}-P)=0$, $\pi_3(\mathbf 1 -Q)\ne 0$;
\item $\pi_4 (\mathbf{1}-Q)=0$, $\pi_4 (\mathbf 1- P)=0$;
\end{itemize}
where any of $\mathcal H_j$, $j=1,2,3,4$, could be zero.
\end{theorem}

\begin{proof}
We will use the fact that any representation of $\mathcal{O}_n^{(0)}$ is a direct sum of a multiple of the Fock representation and a representation of $\mathcal O_n$.

So, restrict $\pi$ to $\mathcal{O}_n^{(0)}\subset\mathcal E_{n,m}^q$, and decompose $\mathcal H=\mathcal H_{F}\oplus\mathcal H_F^{\perp}$, where
\[
\pi(\mathbf 1-Q)_{|\mathcal H_F^{\perp}}=0,
\]
and $\pi(\mathcal{O}_n^0)_{|\mathcal{H}_F}$ is a multiple of the Fock representation.
Denote
\[
S_j := \pi(s_j)\restriction_{\mathcal H_F}, \quad Q := \pi(Q)\restriction_{\mathcal H_F}.
\]
Since
\[
\mathcal H_F=\bigoplus_{\lambda\in\Lambda_n} S_{\lambda} (\ker Q),
\]
it is invariant with respect to $\pi(t_r)$, $\pi(t_r^*)$, $r=\overline{1,m}$. Indeed, $t_r Q=Q t_r$ in $\mathcal E_{n,m}^q$, implying the invariance of $\ker Q$ with respect to $\pi(t_r)$ and $\pi(t_r^*)$. Denote $\ker Q$ by $\mathcal{G}$ and $T_r := \pi(t_r)\restriction_{\mathcal G}$. Then
\[
\pi(t_r)S_{\lambda}\xi=q^{|\lambda|}S_{\lambda}\pi(t_r)\xi=q^{|\lambda|} S_{\lambda}T_r \xi, \quad \xi \in \mathcal G.
\]
Thus $\mathcal{H}_F\simeq\mathcal{F}_n\otimes\mathcal G$ with
\[
\pi(s_j)\restriction_{\mathcal{H}_F}=\pi_{F,n}(s_j)\otimes\mathbf{1}_{\mathcal G},\quad
\pi(t_r)\restriction_{\mathcal{H}_F}=d_n(q)\otimes T_r,\quad j=\overline{1,n},\ r=\overline{1,m},
\]
where the family $\{T_r\}$ determines a $*$-representation $\widetilde{\pi}$ of $\mathcal O_m^{(0)}$ on $\mathcal G$.

Further, decompose $\mathcal G$ as $\mathcal G=\mathcal G_{F}\oplus\mathcal G_F^{\perp}$ into an orthogonal sum of subspaces invariant with respect to $\widetilde{\pi}$, where $\mathcal G_F=\mathcal F_m\otimes\mathcal{K}$,
\[
\widetilde{\pi}_{\mathcal G_F}(t_r)=\pi_{F,m}(t_r)\otimes\mathbf{1}_{\mathcal{K}},\quad r=\overline{1,m},\quad  \mbox{and}\quad
\widetilde{\pi}\restriction_{\mathcal G_F^{\perp}}(\mathbf{1}-P)=0.
\]
Thus $\mathcal H_F=\left(\mathcal{F}_n\otimes\mathcal F_m\otimes\mathcal K\right)\oplus\left(\mathcal F_n\otimes\mathcal G_F^{\perp}\right)$ and
\begin{align*}
\pi_{\mathcal H_F}(s_j)&=\left(\pi_{F,n}(s_j)\otimes\mathbf{1}_{\mathcal F_m}\otimes\mathbf{1}_K \right)\oplus ( \pi_{F,n}(s_j)\otimes\mathbf{1}_{\mathcal G_F^{\perp}} ),\quad j=\overline{1,n},
\\
\pi_{\mathcal H_F}(t_r)&=\left(d_n(q)\otimes\pi_{F,m}(t_r)\otimes\mathbf{1}_K\right)\oplus
\left(d_n(q)\otimes\widetilde{\pi}_{|\mathcal{G}_F^{\perp}}(t_r)\right),\quad r=\overline{1,m}.
\end{align*}
Put $\mathcal H_1=\mathcal F_n\otimes\mathcal F_m\otimes\mathcal{K}=\mathcal F\otimes\mathcal K$ and notice that that $\pi\restriction_{\mathcal{H}_1}=\pi_F^q\otimes\mathbf{1}_{\mathcal K}$, see Remark \ref{rem_fock}. Put $\mathcal H_3=\mathcal F_n\otimes\mathcal G_F^{\perp}$ and $\pi_3=\pi\restriction_{\mathcal H_3}$ i.e.,
\[
\pi_3(s_j)=\pi_{F,n}(s_j)\otimes\mathbf{1}_{\mathcal{G}_F^\perp},\quad
\pi_3(t_r)=d_n(q)\otimes \widetilde{\pi}_{|\mathcal{G}_F^{\perp}}(t_r),\quad
j=\overline{1,n},\ r=\overline{1,m}.
\]
Evidently, $\pi_3(\mathbf 1 -P)=0$ and $\pi_3(\mathbf 1-Q)\ne 0$.

Finally, applying similar arguments to the invariant subspace $\mathcal H_F^{\perp}$ one can show that there exists a decomposition
\[
\mathcal H_F^{\perp}=\mathcal H_2\oplus \mathcal H_4
\]
into the orthogonal sum of invariant subspaces, where
\begin{itemize}
\item
$\mathcal H_2=\mathcal F_m\otimes\mathcal L$ and
\[
\pi_2(s_j) := \pi\restriction_{\mathcal H_2}(s_j)= d_m (\overline{q})\otimes\widehat{\pi}(s_j),\quad
\pi_2(t_r):= \pi\restriction_{\mathcal H_2}(t_r) = \pi_{F,m}(t_r)\otimes\mathbf 1_{\mathcal L},
\]
for a representation $\widehat{\pi}$ of $\mathcal O_n$. Evidently,
$\pi_2(\mathbf 1 - Q)=0$, $\pi_2 (\mathbf 1 - P)\ne 0$.
\item For $\pi_4:=\pi\restriction_{\mathcal H_4}$ one has
\[
\pi_4(\mathbf 1 -Q)=0,\quad \pi_4(\mathbf 1 - P)=0.
\]
\end{itemize}
\end{proof}

\subsection{Ideals in $\mathcal{E}_{n,m}^q$}
In this part, we give a complete description of ideals in $\mathcal{E}_{n,m}^q$, and prove their independence on the deformation parameter $q$.

For
\[
 Q=\sum_{j=1}^n s_j s_j^*,\quad P=\sum_{r=1}^m t_r t_r^*.
\]
we consider  two-sided ideals, $\mathcal M_q$ generated by $1 - P$ and $1 - Q$, $\mathcal I_1^q$ generated by $1-Q$, $\mathcal I_2^q$~generated by $\mathbf 1-P$, and $\mathcal I_q$ generated by $(\mathbf 1-Q)(\mathbf 1-P)$. Evidently, \[ \mathcal I_q = \mathcal I^q_1 \cap \mathcal I^q_2 = \mathcal I^q_1 \cdot \mathcal I^q_2. \] Below we will show that any ideal in $\mathcal E_{n,m}^q$ coincides with the one listed above.

To clarify the structure of $\mathcal I_1^q$, $\mathcal I_2^q$ and $\mathcal I_q$, we use the construction of twisted tensor product of a certain $C^*$-algebra with the algebra of compact operators $\mathbb K$, see \cite{weber}. We give a brief review of the construction, adapted to our situation.

Recall that the $C^*$-algebra $\mathbb K$ can be considered as a universal $C^*$-algebra generated by a closed linear span of elements
$e_{\mu\nu}$, $\mu,\nu\in \Lambda_m$ subject to the relations
\[
e_{\mu_1\nu_1}e_{\mu_2\nu_2}=\delta_{\mu_2\nu_1}e_{\mu_1\nu_2},\quad
e_{\mu_1\nu_1}^*=e_{\nu_1 \mu_1},\quad \nu_i,\mu_i\in\Lambda_m,
\]
here $e_{\emptyset}:=e_{\emptyset\emptyset}$ is a minimal projection.

\begin{definition}
Let $A$ be a $C^*$-algebra,
\[
\alpha=\{\alpha_{\mu},\ \mu\in\Lambda_m\}\subset \mathsf{Aut}(A),\ \mbox{ where}\ \alpha_{\emptyset}=\id_A,
\] and $e_{\mu\nu}$, $\mu,\nu\in\Lambda_m$ be the generators of $\mathbb K$ specified above. Construct the universal $C^*$-algebra
\[
\langle A, \mathbb{K} \rangle_\alpha = C^* (a \in A, e_{\mu\nu} \in \mathbb{K} \, |\, \ a e_{\mu\nu} = e_{\mu\nu} \alpha_{\nu}^{-1}(\alpha_{\mu} (a)).
\]
We define $A \otimes_\alpha \mathbb{K}$ as a subalgebra of $\langle A, \mathbb{K} \rangle_\alpha$ generated by $ax$, $a \in A \subset \langle A, \mathbb{K} \rangle_\alpha$, $x \in \mathbb{K} \subset \langle A, \mathbb{K} \rangle_\alpha$.
\end{definition}

Notice that $\langle A, \mathbb{K} \rangle_\alpha$ exists for any $C^*$-algebra $A$ and family $\alpha\subset\mathsf{Aut} (A)$, see \cite{weber}.
\begin{remark}\ \\ \noindent
$1$. Let $x_{\mu}=e_{\mu\emptyset}$. Then $a x_{\mu}=x_{\mu}\alpha_{\mu}(a)$, $a x_{\mu}^*=x_{\mu}^*\alpha_{\mu}^{-1}(a)$, $a\in A$, compare with \cite{weber}.\\ \noindent
$2$. For any $a\in A$ one has $e_{\mu\nu} a=\alpha_{\mu}^{-1}(\alpha_{\nu}(a))e_{\mu\nu}$ implying that
\[
(a e_{\mu\nu})^*=\alpha_{\mu}^{-1}(\alpha_{\nu}(a))e_{\nu\mu}.
\]  \noindent
$3$. For any $a_1,a_2\in A$ one has $(a_1 e_{\mu_1\nu_1})(a_2 e_{\mu_2\nu_2})=\delta_{\nu_1\mu_2} a_1\alpha_{\mu_1}^{-1}(\alpha_{\mu_2}(a_2))e_{\mu_1\nu_2}$.
\end{remark}

\begin{proposition}[\cite{weber}]
Let $A$ be a $C^*$-algebra and
\[
\alpha = \{\alpha_{\mu},\ \mu \in \Lambda_m\} \subset \mathsf{Aut}(A)\ \mbox{with}\ \alpha_{\emptyset} = \id_A.
\]
Then the correspondence
\[
a e_{\mu\nu}\mapsto \alpha_{\mu}(a)\otimes e_{\mu\nu},\quad a\in A,\ \mu,\nu\in\Lambda_m
\]
extends by linearity and continuity to an isomorphism
\[
\Delta_\alpha\colon A\otimes_{\alpha} \mathbb K \rightarrow A\otimes \mathbb K,
\]
where $\Delta_{\alpha}^{-1}$ is constructed via the correspondence
\[
a\otimes e_{\mu\nu}\mapsto \alpha_{\mu}^{-1}(a) e_{\mu\nu},\quad a\in A, \ \mu,\nu\in\Lambda_m.
\]
\end{proposition}

\begin{remark}
For $x_{\mu}=e_{\mu\emptyset}$, $\mu\in\Lambda_m$ one has, see \cite{weber},
\[
\Delta_{\alpha}(a x_\mu )= \alpha_\mu(a) \otimes x_\mu,\quad
\Delta_{\alpha}(a x_\mu^* )= a \otimes x_\mu^*.
\]
\end{remark}

The following functorial property of $\otimes_{\alpha}\mathbb K$ can be derived easily.
Consider \[ \alpha = (\alpha_\mu)_{\mu \in\Lambda_m} \subset \mathsf{Aut}(A), \ \beta = (\beta_\mu)_{\mu \in\Lambda_m} \subset \mathsf{Aut}(B). \] Suppose $\varphi : A \rightarrow B$ is equivariant, i.e.  $\varphi(\alpha_\mu(a)) = \beta_\mu(\varphi(a))$ for any $a\in A$ and $\mu\in\Lambda_m$. Then one can define the homomorphism
\[
\varphi \otimes_\alpha^\beta : A \otimes_\alpha \mathbb{K} \rightarrow B \otimes_\beta \mathbb{K}, \quad \varphi\otimes_\alpha^\beta(ak)= \varphi(a)k,\quad a\in A,\ k\in\mathbb K,
\]
making the following diagram commutative
\begin{equation}\label{ktwist_func}
\begin{tikzcd}
A \otimes_\alpha \mathbb{K} \arrow[d, "\Delta_\alpha"] \arrow[r, "\varphi \otimes_\alpha^\beta"] & B \otimes_\beta \mathbb{K} \arrow[d, "\Delta_\beta"] \\
A \otimes \mathbb{K} \arrow[r, "\varphi \otimes \id_\mathbb{K}"] & B \otimes \mathbb{K}
\end{tikzcd}
\end{equation}
Namely, it is easy to verify that
\[
(\Delta_\beta^{-1}\circ(\varphi\otimes \id_{\mathbb K})\circ\Delta_\alpha) (a e_{\mu\nu})=\varphi(a)e_{\mu\nu}=\varphi\otimes_{\alpha}^{\beta}(a e_{\mu\nu}),\quad a\in A,\ \mu,\nu\in\Lambda_m.
\]

An important consequence of the commutativity of the diagram above is exactness of the functor $\otimes_{\alpha}\mathbb K$. Let
\[
\beta=(\beta_{\mu})_{\mu\in\Lambda_m}\subset\mathsf{Aut} (B),\
\alpha=(\alpha_{\mu})_{\mu\in\Lambda_m}\subset\mathsf{Aut} (A),\  \gamma=(\gamma_{\mu})_{\mu\in\Lambda_m}\subset\mathsf{Aut} (C)
\]
and consider a short exact sequence
\[
\begin{tikzcd}
 0 \arrow[r] & B \arrow[r, "\varphi_1"]  & A \arrow[r, "\varphi_2"]& C\arrow[r] & 0
 \end{tikzcd}
\]
where $\varphi_1$, $\varphi_2$ are equivariant homomorphisms. Then one has the following short exact sequence
\[
\begin{tikzcd}
 0 \arrow[r] & B\otimes_{\beta}\mathbb K \arrow[r, "\varphi_1\otimes_\beta^\alpha"]  & A\otimes_{\alpha}\mathbb K \arrow[r, "\varphi_2\otimes_\alpha^\gamma"]& C\otimes_\gamma \mathbb K\arrow[r] & 0
 \end{tikzcd}
\]

Now we are ready to study the structure of the ideals $\mathcal I_1^q\, ,\mathcal I_2^q,\mathcal I_q\subset\mathcal{E}_{n,m}^q$. We start with $\mathcal I_1^q$. Notice that
\[
\mathcal I_1^q=c.l.s.\, \{ \,t_{\mu_2}t_{\nu_2}^* s_{\mu_1}(\mathbf 1-Q)s_{\nu_1}^*,\ \mu_1,\nu_1\in\Lambda_n,\ \mu_2,\nu_2\in\Lambda_m\}.
\]
Put $E_{\mu_1\nu_1}=s_{\mu_1}(\mathbf 1-Q)s_{\nu_1}^*$, $\mu_1,\nu_1\in\Lambda_n$. Then $E_{\mu_1\nu_1}$ satisfy the relations for matrix units generating $\mathbb K$.  Moreover, $c.l.s.\,\{E_{\mu\nu},\ \mu,\nu\in\Lambda_n\}$ is an ideal in $\mathcal O_n^{(0)}$ isomorphic to $\mathbb K$.

Consider the family $\alpha^q =(\alpha_\mu)_{ \mu\in\Lambda_n}\subset\mathsf{Aut}(\mathcal O_m^{(0)})$ defined as
\[
\alpha_{\mu} (t_r)=q^{|\mu|} t_r,\quad
\alpha_{\mu} (t_r^*)=q^{-|\mu|} t_r^*,\quad \mu\in\Lambda_n,\ r=\overline{1,m}.
\]
\begin{proposition}\label{str_I1q}
The correspondence $a e_{\mu\nu}\mapsto a E_{\mu\nu}$, $a\in \mathcal O_m^{(0)}$, $\mu,\nu\in\Lambda_n$, extends to an isomorphism
\[
\Delta_{q,1} \colon \mathcal O_m^{(0)}\otimes_{\alpha^q}\mathbb K\rightarrow \mathcal I_1^q.
\]
\end{proposition}
\begin{proof}
We note that for any $\mu_1,\nu_1\in\Lambda_n$ and $\mu_2,\nu_2\in\Lambda_m$ one has
\[
t_{\mu_2}t_{\nu_2}^* E_{\mu_1\nu_1}= q^{(|\nu_1|-|\mu_1|)(|\mu_2|-|\nu_2|)}
E_{\mu_1\nu_1}t_{\mu_2}t_{\nu_2}^* = E_{\mu_1\nu_1}
\alpha_{\nu_1}^{-1}(\alpha_{\mu_1}(t_{\mu_2}t_{\nu_2}^*)).
\]
Thus, due  to the universal property of $\langle \mathcal O_m^{(0)}, \mathbb K \rangle_{\alpha^q}$, the correspondence \[ a e_{\mu \nu} \mapsto a E_{\mu \nu} \] determines a surjective homomorphism
$\Delta_{q,1}\colon \mathcal O_m^{(0)}\otimes_{\alpha^q}\mathbb K\rightarrow \mathcal I_1^q$.

It remains to show that $\Delta_{q,1}$ is injective. Since the Fock representation of $\mathcal E_{n,m}^q$ is faithful, we can identify $\mathcal I_1^q$ with $\pi_F^q (\mathcal I_1^q)$. It will be convenient for us to use the following form of the Fock representation, see Remark \ref{rem_fock},
\begin{align*}
\pi_F^q (s_j)&=\pi_{F,n}(s_j)\otimes\mathbf{1}_{\mathcal F_m}
:= S_j\otimes\mathbf 1_{\mathcal F_m},\ j=\overline{1,n}, \\
\pi_F^q (t_r)&= d_n(q)\otimes\pi_{F,m}(t_r)
:= d_n(q)\otimes T_r,\ r=\overline{1,m}.
\end{align*}
In particular, for any $\mu_1,\nu_1\in\Lambda_n,\ \mu_2,\nu_2\in\Lambda_m$
\[
\pi_F^q (t_{\mu_2}t_{\nu_2}^* E_{\mu_1\nu_1})=
d_n(q^{|\mu_2|-|\nu_2|}) S_{\mu_1} (\mathbf 1 - Q)S_{\nu_1}\otimes T_{\mu_2}T_{\nu_2}^*.
\]

Consider $\Delta_{q,1}\circ\Delta_{\alpha^q}^{-1}\colon \mathcal O_m^{(0)}\otimes\mathbb K\rightarrow \pi_F^q (\mathcal I_1^q)$. We intend to show that
\[
\Delta_{q,1}\circ\Delta_{\alpha^q}^{-1}=\pi_F^1,
\]
where $\pi_F^1$ is the restriction of the Fock representation of $\mathcal{O}_n^{(0)}\otimes\mathcal{O}_m^{(0)}$ to $\mathbb K\otimes\mathcal O_m^{(0)}$, and $\mathbb K$ is generated by $E_{\mu\nu}$ specified above. Notice that the family
\[
\{t_{\mu_2}t_{\nu_2}^*\otimes E_{\mu_1\nu_1},\ \mu_1,\nu_1\in\Lambda_n,\ \mu_2,\nu_2\in\Lambda_m\}
\]
generates $\mathcal{O}_m^{(0)}\otimes\mathbb K$. Then
\[
\Delta_{\alpha^q}^{-1}(t_{\mu_2}t_{\nu_2}^*\otimes E_{\mu_1\nu_1})=
\alpha_{\mu_1}^{-1}(t_{\mu_2}t_{\nu_2}^*) e_{\mu_1\nu_1}=
q^{-|\mu_1|(|\mu_2|-|\nu_2|)}t_{\mu_2}t_{\nu_2}^* e_{\mu_1\nu_1},
\]
and
\begin{align*}
\Delta_{q,1}\circ & \Delta_{\alpha^q}^{-1}(t_{\mu_2}t_{\nu_2}^*\otimes E_{\mu_1\nu_1})=
q^{-|\mu_1|(|\mu_2|-|\nu_2|)}\pi_F^q(t_{\mu_2}t_{\nu_2}^* E_{\mu_1\nu_1})
\\
&=q^{-|\mu_1|(|\mu_2|-|\nu_2|)}d_n(q^{|\mu_2|-|\nu_2|})S_{\mu_1}
(\mathbf 1-Q)S_{\nu_1}^*\otimes T_{\mu_2}T_{\nu_2}^*
\\
&=
q^{-|\mu_1|(|\mu_2|-|\nu_2|)}q^{|\mu_1|(|\mu_2|-|\nu_2|)}S_{\mu_1}d_n(q^{|\mu_2|-|\nu_2|})
(\mathbf 1-Q)S_{\nu_1}^*\otimes T_{\mu_2}T_{\nu_2}^*
\\
&=S_{\mu_1}(\mathbf 1-Q)S_{\nu_1}^*\otimes T_{\mu_2}T_{\nu_2}^*=
\pi_F^1 (E_{\mu_1\nu_1}\otimes t_{\mu_2}t_{\nu_2}^*),
\end{align*}
where we used relations $d_n(\lambda)S_j=\lambda S_j d_n(\lambda)$, $j=\overline{1,n}$, $\lambda\in\mathbb C$, and the obvious fact that
\[
d_n(\lambda) (\mathbf 1 -Q)=\mathbf 1 -Q.
\]

To complete the proof we recall that $\pi_F^1$ is a faithful representation of $\mathcal O_n^{(0)}\otimes\mathcal O_{m}^0$, so its restriction to $\mathbb K\otimes\mathcal O_m^{(0)}$ is also faithful, implying the injectivity of $\Delta_q$.
\end{proof}

\begin{remark}\label{rem_iq_1}
Evidently, $\mathcal I_q$ is a closed linear span of the family
\[
\{\, t_{\mu_2}(1-P)t_{\nu_2}^*s_{\mu_1}(\mathbf 1 -Q)s_{\nu_1}^*,\
\mu_1,\nu_1\in\Lambda_n,\ \mu_2,\nu_2\in\Lambda_m\}\subset \mathcal I_1^q.
\]
Moreover, $c.l.s.\{t_{\mu_2}(1-P)t_{\nu_2}^*,\  \mu_2,\nu_2\in\Lambda_m\}=\mathbb K\subset\mathcal O_m^{(0)}$. It is easy to see that
\[
\alpha_{\mu}(t_{\mu_2}(1-P)t_{\nu_2}^*)=q^{|\mu|(|\mu_2|-|\nu_2|)}
t_{\mu_2}(1-P)t_{\nu_2}^*,
\]
so every $\alpha_\mu\in \alpha^q$ can be regarded as an element of $\mathsf{Aut} (\mathbb K)$.
\end{remark}

A moment reflection and Proposition \ref{str_I1q} give the following corollary
\begin{proposition}\label{str_Iq_1}
Restriction of $\Delta_{q,1}$ to $\mathbb{K}\otimes_{\alpha^q}\mathbb K \subset \mathcal{O}_m^{(0)} \otimes_{\alpha^q} \mathbb{K}$ gives an isomorphism \[ \Delta_{q,1} \colon \mathbb{K}\otimes_{\alpha^q}\mathbb K \rightarrow \mathcal I_q . \]
\end{proposition}

To deal with $\mathcal I_2^q$, we consider the family $\beta^q=\{\beta_{\mu},\ \mu\in\Lambda_m\}\subset \mathsf{Aut}(\mathcal O_n^{(0)})$ defined as
\[
\beta_\mu (s_j)=q^{-|\mu|}s_j,\ \beta_{\mu}(s_j^*)=q^{|\mu|}s_j^*,\ j=\overline{1,n}.
\]
\begin{proposition}\label{str_I2q}
One has an isomorphism $\Delta_{q,2}\colon\mathcal{O}_n^{(0)}\otimes_{\beta^q}\mathbb K\rightarrow \mathcal I_2^q$.
\end{proposition}

Obviously, $\Delta_{q,2}$ induces the isomorphism
$\mathbb K\otimes_{\beta^q}\mathbb K\simeq\mathcal I_q$, where the first term is an ideal in $\mathcal O_n^{(0)}$ and the second in $\mathcal O_m^{(0)}$ respectively.

Write
\[
\varepsilon_n\colon\mathbb K\rightarrow \mathcal O_n^{(0)}, \quad \varepsilon_m\colon\mathbb K\rightarrow \mathcal O_m^{(0)},
\]
for the canonical embeddings and
\[
q_n\colon\mathcal O_n^{(0)}\rightarrow\mathcal O_n, \quad q_m\colon\mathcal O_m^{(0)}\rightarrow \mathcal{O}_m,
\]
for the quotient maps. Let also
\[
\varepsilon_{q,j}\colon\mathcal I_q\rightarrow\mathcal I_{j}^q, \quad j=1,2,
\]
be the embeddings and
\[
\pi_{q,j}\colon\mathcal I_{j}^q\rightarrow \mathcal I_j^q/\mathcal I_q, \quad j=1,2,
\]
the quotient maps. Notice also that the families $\alpha^q\subset\mathsf{Aut}(\mathcal O_m^{(0)})$,
$\beta^q\subset\mathsf{Aut}(\mathcal O_n^{(0)})$ determine families of automorphisms of $\mathcal O_m$ and $\mathcal O_n$ respectively, also denoted by $\alpha^q$ and $\beta^q$.

\begin{theorem}\label{i1i2q_comm}
One has the following isomorphism of extensions
\[
\begin{tikzcd}
 0 \arrow[r] & \mathcal{I}_q \arrow[r, "\varepsilon_{q,1}"] \arrow[d,"\Delta_{\alpha^q}\circ\Delta_{q,1}^{-1}"] & \mathcal{I}_1^q \arrow[d, "\Delta_{\alpha^q}\circ\Delta_{q,1}^{-1}"] \arrow[r, "\pi_{q,1}"] & \mathcal{I}_1^q / \mathcal{I}_q \arrow[d, "\simeq"] \arrow[r] & 0 \\
  0 \arrow[r] & \mathbb{K} \otimes \mathbb{K} \arrow[r,  "\varepsilon_m\otimes\id_\mathbb{K}"] &\mathcal{O}_m^0 \otimes \mathbb{K}  \arrow[r, "q_m\otimes\id_{\mathbb K}"] & \mathcal{O}_m\otimes \mathbb{K}  \arrow[r] & 0
\end{tikzcd}
\]
and
\[
\begin{tikzcd}
 0 \arrow[r] & \mathcal{I}_q \arrow[r, "\varepsilon_{q,2}"] \arrow[d,"\Delta_{\beta^q}\circ\Delta_{q,2}^{-1}"] & \mathcal{I}_2^q \arrow[d, "\Delta_{\beta^q}\circ\Delta_{q,2}^{-1}"] \arrow[r,"\pi_{q,2}"] & \mathcal{I}_2^q / \mathcal{I}_q \arrow[d, "\simeq"] \arrow[r] & 0 \\
  0 \arrow[r] & \mathbb{K} \otimes \mathbb{K} \arrow[r, "\varepsilon_n\otimes
  \id_\mathbb{K}"] & \mathcal{O}_n^0\otimes\mathbb{K} \arrow[r, "q_n\otimes\id_{\mathbb K} "] & \mathcal{O}_n\otimes\mathbb{K} \arrow[r] & 0
\end{tikzcd}
\]
\end{theorem}

\begin{proof}
Indeed, each  row in diagram (\ref{comm_diag_tw}) below is exact and every non-dashed vertical arrow  is an isomorphism. The bottom left and bottom right squares are commutative due to (\ref{ktwist_func}). The top left square is commutative due to the arguments in the proof of Proposition \ref{str_I1q} combined with Remark \ref{rem_iq_1}. Hence there exists a unique isomomorphism
 \[
 \Phi_{q,1}\colon\mathcal I_1^q/\mathcal I_q\rightarrow\mathcal O_m\otimes_{\alpha^q}\mathbb K,
 \]
 making the diagram (\ref{comm_diag_tw}) commutative
\begin{equation}\label{comm_diag_tw}
\begin{tikzcd}
 0 \arrow[r] & \mathcal{I}_q \arrow[r, "\varepsilon_{q,1}"] \arrow[d, "\Delta_{q,1}^{-1}"] & \mathcal{I}_1^q \arrow[d, "\Delta_{q,1}^{-1}"] \arrow[r, "\pi_{q,1}"] & \mathcal{I}_1^q / \mathcal{I}_q \arrow[d, dashrightarrow, "\Phi_{q,1}"] \arrow[r] & 0 \\
 0 \arrow[r] & \mathbb{K} \otimes_{\alpha^q} \mathbb{K} \arrow[r, "\varepsilon_{m}\otimes_{\alpha^q}^{\alpha^q}"] \arrow[d, "\Delta_{\alpha^q}"] & \mathcal{O}_m^0 \otimes_{\alpha^q}\mathbb{K}  \arrow[d, "\Delta_{\alpha^q}"] \arrow[r, " q_m \otimes_{\alpha_q}^{\alpha^q}"] & \mathcal{O}_m\otimes_{\alpha^q}\mathbb{K}   \arrow[d, "\Delta_{\alpha^q}"] \arrow[r] & 0 \\
  0 \arrow[r] & \mathbb{K} \otimes \mathbb{K} \arrow[r, "\varepsilon_m\otimes\id_\mathbb{K} "] & \mathcal{O}_m^0\otimes\mathbb{K}  \arrow[r, "q_m\otimes\id_{\mathbb K} "] & \mathcal{O}_m\otimes\mathbb{K} \arrow[r] & 0 \\
\end{tikzcd}
\end{equation}
The proof for $\mathcal I_2^q$ is similar.
\end{proof}

The following Lemma follows from the fact that $\mathcal M_q = \mathcal{I}_1^q + \mathcal{I}_2^q$.

\begin{lemma}\label{lemextmq}
\[
\mathcal{M}_q / \mathcal{I}_q \simeq  \mathcal{I}_1^q/\mathcal{I}_q \oplus \mathcal{I}_2^q/\mathcal{I}_q \simeq \mathcal{O}_m \otimes \mathbb{K}\oplus\mathcal O_n\otimes\mathbb{K}.
\]
\end{lemma}

Theorem \ref{i1i2q_comm} implies that $\mathcal I_q, \mathcal I^q_1, \mathcal I^q_2$ are stable $C^*$-algebras. It follows from \cite{Rordam}, Proposition 6.12, that an extension of a stable $C^*$-algebra by $\mathbb{K}$ is also stable. Thus, Lemma \ref{lemextmq} implies immediately the following important corollary.

\begin{corollary}
For any $q\in\mathbb{C}$, $|q|=1$, the $C^*$-algebra $\mathcal{M}_q$ is stable.
\end{corollary}

Denote the Calkin algebra by $Q$. Recall that for $C^*$-algebras $A$ and $B$ the isomorphism
\[
\mathsf{Ext}( A \oplus B,\mathbb{K}) \simeq \mathsf{Ext}(A,\mathbb{K}) \oplus \mathsf{Ext}(B,\mathbb{K})
\]
is given as follows. Let
\[
\iota_1 : A \rightarrow A \oplus B,\quad  \iota_1(a) = (a,0),\quad
\iota_2 : B \rightarrow A \oplus B,\quad \iota_2(b) = (0,b).
\]
For a Busby invariant $\tau : A \oplus B \rightarrow Q$ define
\[\mathsf{F} : \mathsf{Ext}(A \oplus B, \mathbb{K}) \rightarrow \mathsf{Ext}(A, \mathbb{K}) \oplus \mathsf{Ext}(B, \mathbb{K}),\quad
\mathsf{F}(\tau) = (\tau \circ \iota_1, \tau \circ \iota_2).
\]
It can be shown, see \cite{higson}, that $\mathsf{F}$ determines a group isomorphism.

\begin{remark}\label{rem_beta}
Consider an extension
\begin{equation}\label{ext_busby}
\begin{tikzcd}
  0 \arrow[r] & B \arrow[r] & E \arrow[r] & A \arrow[r]  & 0
\end{tikzcd}
\end{equation}
Let $i\colon B \to M(B)$ be the canonical embedding. Define $\beta$ to be the unique map such that
\[
\beta(e)i(b) = i(eb),\quad\mbox{for every}\ b \in B,\ e\in E.
\]
Then the Busby invariant $\tau$ is the unique map which makes the diagram commute.
\[
\begin{tikzcd}
 0 \arrow[r] & B \arrow[r, "i"] & M(B) \arrow[r] & M(B)/B \arrow[r] & 0 \\
 0 \arrow[r] & B \arrow[r]\arrow[equal]{u} & E \arrow[r]  \arrow[u, "\beta"]& A \arrow[r] \arrow[u, "\tau"] & 0
\end{tikzcd}
\]
We will use both notations $[E]$ and $[\tau]$ in order to denote the class of the extension (\ref{ext_busby}) in $\mathsf{Ext}(A, B)$.
\end{remark}

Let $[\mathcal{M}_q] \in \mathsf{Ext}( \mathcal{I}_1^q/\mathcal{I}_q \oplus \mathcal{I}_2^q/\mathcal{I}_q, \mathcal I_q)$, $[\mathcal I_1^q] \in \mathsf{Ext}(\mathcal I_1^q / \mathcal{I}_q, \mathcal{I}_q)$, $[\mathcal I_2^q] \in \mathsf{Ext}(\mathcal I_2^q / \mathcal{I}_q, \mathcal{I}_q)$ respectively be the classes of the following extensions
\begin{gather*}
0 \rightarrow \mathcal I_q \rightarrow \mathcal M_q \rightarrow \mathcal{I}_1^q/\mathcal{I}_q \oplus \mathcal{I}_2^q/\mathcal{I}_q \rightarrow 0,
\\
0 \rightarrow \mathcal I_q \rightarrow \mathcal I^q_1 \rightarrow \mathcal I^q_1 / \mathcal I_q \rightarrow 0,
\\
0 \rightarrow \mathcal I_q \rightarrow \mathcal I^q_2 \rightarrow \mathcal I^q_2 / \mathcal I_q \rightarrow 0.
\end{gather*}

\begin{lemma}\label{lem_qdecomp}
\[
[\mathcal{M}_q] = ([\mathcal{I}_1^q], [\mathcal{I}_2^q]) \in \mathsf{Ext}(\mathcal I_1^q / \mathcal{I}_q, \mathcal{I}_q) \oplus \mathsf{Ext}(\mathcal I_2^q / \mathcal{I}_q, \mathcal{I}_q) \simeq \mathsf{Ext}( \mathcal{I}_1^q/\mathcal{I}_q \oplus \mathcal{I}_2^q/\mathcal{I}_q,\mathcal{I}_q).
\]
\end{lemma}

\begin{proof}
Consider the following morphism of extensions:
\[
\begin{tikzcd}[row sep=scriptsize, column sep=scriptsize]
& \mathcal{I}_q \arrow[dl, equal] \arrow[rr, "i"] \arrow[dd, equal] & & M(\mathcal{I}_q) \arrow[dd, equal] \arrow[rr] & & M(\mathcal{I}_q)/\mathcal{I}_q \arrow[dd, equal] \\
\mathcal{I}_q \arrow[rr, crossing over] \arrow[dd, equal] & & \mathcal{I}_1^q \arrow[ur, "\beta_1"] \arrow[rr,  crossing over] & & \mathcal{I}_1^q/\mathcal{I}_q  \arrow[ur, "\tau_{\mathcal{I}_1^q}"] \\
& \mathcal{I}_q \arrow[dl, equal] \arrow[rr] & & M(\mathcal{I}_q)  \arrow[rr] & & M(\mathcal{I}_q) / \mathcal{I}_q \arrow[from=dl, "\tau_{\mathcal{M}_q}"] \\
\mathcal{I}_q \arrow[rr] & & \mathcal{M}_q \arrow[ur, "\beta_2"] \arrow[from=uu, crossing over, hookrightarrow] \arrow[rr] & & \mathcal{I}_1^q/\mathcal{I}_q \oplus \mathcal{I}_2^q/\mathcal{I}_q \arrow[from=uu, crossing over] \\
\end{tikzcd}
\]
Here
\[
\beta_1\colon\mathcal I_1^q \rightarrow M(\mathcal I_q),\quad \beta_2\colon\mathcal M_q \rightarrow M(\mathcal I_q),
\]
are homomorphisms introduced in Remark \ref{rem_beta}, the arrow
\[
j_1\colon \mathcal{I}_1^q \hookrightarrow \mathcal{M}_q
\]
is the inclusion, and the arrow
\[
\iota_1 : \mathcal{I}_1^q / \mathcal{I}_q \rightarrow \mathcal{I}_1^q / \mathcal{I}_q \oplus \mathcal{I}_2^q / \mathcal{I}_q
\]
has the form $\iota_1(x) = (x, 0)$.

Notice that for every $b \in  \mathcal{I}_q$ and
$x \in \mathcal{I}_1^q$ one has
\[
(\beta_2 \circ j_1)(x) i(b) = i(j_1(x)b) = i(xb)=\beta_1(x)i(b).
\]
By the uniqueness of $\beta_1$, we get $\beta_2 \circ j_1 = \beta_1$. Thus the following diagram commutes
\[
\begin{tikzcd}
\mathcal I_1^q \arrow[r, "\beta_1"] \arrow[d, hookrightarrow] & M(\mathcal I_q) \arrow[d, equal] \\
 \mathcal M_q \arrow[r, "\beta_2"] & M(\mathcal I_q)
\end{tikzcd}
\]
Further, Remark \ref{rem_beta} implies that for  Busby invariants $\tau_{\mathcal I_1^q}$ and $\tau_{\mathcal M_q}$  the squares below are commutative
\[
\begin{tikzcd}
 M(\mathcal I_q) \arrow[r] & M(\mathcal I_q)/\mathcal I_q \\
 \mathcal I_1^q \arrow[r]  \arrow[u, "\beta_1"]& \mathcal I_1^q/\mathcal I_q \arrow[u, "\tau_{\mathcal I_1^q}"]
\end{tikzcd}\quad
\begin{tikzcd}
 M(\mathcal I_q) \arrow[r] & M(\mathcal I_q)/\mathcal I_q \\
 \mathcal M_q \arrow[r]  \arrow[u, "\beta_2"]& \mathcal I_1^q/\mathcal I_q
 \oplus \mathcal I_2^q/\mathcal I_q\arrow[u, "\tau_{\mathcal M_q}"]
\end{tikzcd}
\]
Hence the square
\[
\begin{tikzcd}
 \mathcal I_1^q/\mathcal I_q \arrow[r,"\tau_{\mathcal I_1^q}"] \arrow[d,"\iota_1"]& M(\mathcal I_q)/\mathcal I_q \arrow[d,equal]\\
 \mathcal I_1^q/\mathcal I_q\oplus\mathcal I_2^q/\mathcal I_q \arrow[r, "\tau_{\mathcal M_q}"] & M(\mathcal I_q)/\mathcal I_q
\end{tikzcd}
\]
is also commutative. Thus,
$\tau_{\mathcal{I}_1^q} = \tau_{\mathcal{M}_q} \circ \iota_1$. By the same arguments we get $\tau_{\mathcal I_2^q}=\tau_{\mathcal M_q}\circ\iota_2$, where
\[
\iota_2\colon\mathcal I_2^q/\mathcal I_q\rightarrow \mathcal{I}_1^q / \mathcal{I}_q \oplus \mathcal{I}_2^q / \mathcal{I}_q,\quad \iota_2 (y)=(0,y).
\]
Thus
\[ [\tau_{\mathcal{M}_q}] = ([\tau_{\mathcal{M}_q}\circ\iota_1], [\tau_{\mathcal{M}_q} \circ\iota_2]) = ([\tau_{\mathcal{I}_1^q}], [\tau_{\mathcal{I}_2^q}]).
\]
\end{proof}

In the following theorem we give a description of all ideals in $\mathcal E_{n,m}^q$.

\begin{theorem}\label{ideals_enmq}
Any ideal $J\subset\mathcal E_{n,m}^q$ coincides with one of $\mathcal I_q$,
$\mathcal I_1^q$, $\mathcal I_2^q$, $\mathcal M_q$.
\end{theorem}

\begin{proof}
First we notice that $\mathcal I_1^q/\mathcal I_q\simeq \mathcal O_m\otimes\mathbb K$, $\mathcal I_2^q/\mathcal I_q\simeq \mathcal O_n\otimes\mathbb K$ are simple. Hence for any ideal $\mathcal J$ such that $\mathcal I_q\subseteq\mathcal J\subseteq\mathcal I_1^q$ or  $\mathcal I_q\subseteq\mathcal J\subseteq\mathcal I_2^q$,
one has  $\mathcal J=\mathcal I_q$, or $\mathcal J=\mathcal I_1^q$, or $\mathcal J=\mathcal I_2^q$.

Further, using the fact that $\mathcal M_q=\mathcal I_1^q+\mathcal I_2^q$ and $\mathcal I_q=\mathcal I_1^q\cap\mathcal I_2^q$ we get
\[
\mathcal M_q/\mathcal I_1^q\simeq\mathcal I_2^q/\mathcal I_q\simeq\mathcal O_n\otimes\mathbb K.
\]
So if $\mathcal I^q \subseteq \mathcal J\subseteq\mathcal M_q$, then again either
$\mathcal J=\mathcal I_1^q$ or $\mathcal J=\mathcal M_q$.

Below, see Theorem \ref{Onmqsimple}, we show that $\mathcal E_{n,m}^q/\mathcal M_q$ is simple and purely infinite. In particular, $\mathcal M_q$ contains any ideal in $\mathcal E_{n,m}^q$, see Corollary \ref{mq_unique}.

Let $\mathcal J\subset \mathcal E_{n,m}^q$ be an ideal and $\pi$ be a representation of $\mathcal{E}_{n,m}^q$ such that $\ker \pi = \mathcal{J}$. Notice that the Fock component $\pi_1$ in the Wold decomposition of $\pi$ is zero. Thus, by Theorem \ref{wold_dec},
\begin{equation}\label{pi_decomp}
\pi=\pi_2\oplus\pi_3\oplus\pi_4,
\end{equation}
and $\mathcal J=\ker\pi=\ker\pi_2\cap\ker\pi_3\cap\ker\pi_4$. Let us describe these kernels. Suppose that the component $\pi_2$ is non-zero. Since $\pi_2(\mathbf 1-Q)=0$ and $\pi_2(\mathbf 1-P)\ne 0$, we have
\[
\mathcal I_1^q\subseteq\ker\pi_2\subsetneq\mathcal M_q,
\]
implying $\ker\pi_2=\mathcal I_1^q$. Using the same arguments, one can deduce that if the component $\pi_3$ is non-zero, then $\ker\pi_3=\mathcal I_2^q$, and if $\pi_4$ is non-zero, then $\ker\pi_4=\mathcal M_q$.

Finally, if in (\ref{pi_decomp}) $\pi_2$ and $\pi_3$ are non-zero then $\mathcal J=\ker\pi=\mathcal I_q$. If either $\pi_2 \neq 0$ and $\pi_3 = 0$ or $\pi_3\neq 0$ and $\pi_2=0$, then either  $\mathcal J=\mathcal I_1^q$ or $\mathcal J=\mathcal I_2^q$. In the case $\pi_2=0$ and $\pi_3=0$ one has $\mathcal J=\ker\pi_4=\mathcal M_q$.
\end{proof}

\begin{corollary}\label{essential}
 All ideals in $\mathcal E_{n,m}^q$ are essential. The ideal $\mathcal I^q$ is the unique minimal ideal.
\end{corollary}

In particular, the extension
\[ 0
\rightarrow \mathcal{I}_q \rightarrow \mathcal{M}_q \rightarrow \mathcal{I}_1^q/\mathcal{I}_q \oplus \mathcal{I}_2^q/\mathcal{I}_q \rightarrow 0
\]
is essential.
Indeed, the ideal $\mathbb K=\mathcal{I}_q\subset \mathcal E_{n,m}^q$ is the unique minimal ideal. Since an ideal of an ideal in a $C^*$-algebra is an ideal in the whole algebra, $\mathcal{I}_q$ is the unique minimal ideal in $\mathcal M_q$, thus it is essential in $\mathcal M_q$.

The following proposition is a corollary of Voiculescu's Theorem, see Theorem 15.12.3 of \cite{Black}.

\begin{proposition}\label{Voiculescu}
Let $E_1, E_2$ be two essential extensions of a nuclear $C^*$-algebra $A$ by $\mathbb{K}$. If $[E_1] = [E_2] \in \mathsf{Ext}(A, \mathbb{K})$ then $E_1 \simeq E_2$.
\end{proposition}

\begin{theorem}
For any $q\in\mathbb C$, $|q|=1$, one has
$\mathcal{M}_q \simeq \mathcal{M}_1$.
\end{theorem}

\begin{proof}
By Theorem \ref{i1i2q_comm}, $[\mathcal{I}_1^q]\in\mathsf{Ext}(\mathcal{O}_m \otimes \mathbb{K},\mathbb{K})$, and  $[\mathcal{I}_2^q] \in \mathsf{Ext}(\mathcal{O}_n \otimes \mathbb{K}, \mathbb{K})$ do not depend on $q$. By Lemma \ref{lem_qdecomp}, $[\mathcal{M}_q]$ does not depend on $q$. Thus by Corollary \ref{essential} and Proposition \ref{Voiculescu}, $\mathcal{M}_q \simeq \mathcal{M}_1$.
\end{proof}

\subsection{Simplicity and pure infiniteness of $\mathcal{O}_n \otimes_q \mathcal{O}_m$}
The next step is to show that the quotient
$\mathcal{O}_n \otimes_q \mathcal{O}_m=\mathcal{E}_{n,m}^q/\mathcal{M}_q$, being nuclear, is also simple and purely infinite.

It is easy to see that
\[
\mathcal{M}_q=c.l.s.\{s_{\mu_1}t_{\nu_1}(\mathbf{1}-P)^{\varepsilon_1}(\mathbf 1- Q)^{\varepsilon_2}t_{\nu_2}^*s_{\mu_2}^*\},
\]
where $\mu_j\in\Lambda_n$, $\nu_j\in\Lambda_m$, $j=1,2$, and $\varepsilon_j\in\{0,1\}$,
$\varepsilon_1+\varepsilon_2\ne 0$.

We denote the generators of $\mathcal{O}_n\otimes_q\mathcal{O}_m$ in the same way as generators of $\mathcal E_{n,m}^q$. Notice, that for any $k\in\mathbb N$,  the following
relations hold in  $\mathcal{O}_n\otimes_q\mathcal{O}_m$
\[
 \sum_{\lambda\in\Lambda_n,\, |\lambda|=k} s_{\lambda} s_{\lambda}^*=\mathbf{1},\quad
 \sum_{\nu\in\Lambda_m,\, |\nu|=k} t_{\nu} t_{\nu}^*=\mathbf{1},
\]
and
\[
\mathcal{O}_n \otimes_q \mathcal{O}_m=c.l.s.\{s_{\mu_1}s_{\mu_2}^* t_{\nu_1}t_{\nu_2}^*,\ \mu_i\in\Lambda_n,\ \nu_j\in\Lambda_m\}.
\]
Consider the action $\alpha$ of $\mathbb{T}^2$ on $\mathcal O_n\otimes_q\mathcal O_m$,
\[
\alpha_{\varphi_1,\varphi_2}(s_j)=e^{2\pi i\varphi_1}s_j,\quad \alpha_{\varphi_1,\varphi_2}(t_r)=e^{2\pi i\varphi_2} t_r,\quad j=\overline{1,n},\ r=\overline{1,m}.
\]
Construct the corresponding faithful conditional expectation $E_{\alpha}$, and
denote by $\mathcal A_q$ the fixed point $C^*$-algebra of $(\mathcal O_n\otimes_q \mathcal O_m)^{\alpha}$, see Section \ref{group_action}.
Similarily to the case of $\mathcal E_{n,m}^q$, one has
\begin{align*}
E_{\alpha}(s_{\mu_1}s_{\mu_2}^*t_{\nu_1}t_{\nu_2}^*)&=0,\quad \mbox{ if either}\  |\mu_1|\ne |\mu_2|\ \mbox{ or}\  |\nu_1|\ne |\nu_2|,\\
E_{\alpha}(s_{\mu_1}s_{\mu_2}^*t_{\nu_1}t_{\nu_2}^*)&=s_{\mu_1}s_{\mu_2}^*t_{\nu_1}t_{\nu_2}^*\quad \mbox{ if}\  |\mu_1|=|\mu_2|\ \mbox{ and}\  |\nu_1|=|\nu_2|.
\end{align*}

\begin{lemma} \label{lem_comm_rel}
If\/ $\nu_1, \nu_2 \in\Lambda_m$, then
\[
s_j t_{\nu_1}t_{\nu_2}^*=\overline{q}^{|\nu_1|-|\nu_2|} t_{\nu_1}t_{\nu_2}^* s_j,\quad
s_j^* t_{\nu_1}t_{\nu_2}^*=q^{|\nu_1|-|\nu_2|}t_{\nu_1}t_{\nu_2}^* s_j^*,\quad
j=\overline{1,n}.
\]
If\/ $\mu_1, \mu_2 \in\Lambda_n$, then
\[
t_i s_{\mu_1}s_{\mu_2}^*=q^{|\mu_1|-|\mu_2|} s_{\mu_1}s_{\mu_2}^* t_i,\quad
t_i^* s_{\mu_1}s_{\mu_2}^*=\overline{q}^{|\mu_1|-|\mu_2|}s_{\mu_1}s_{\mu_2}^* t_i^*,
\quad i=\overline{1,m}.
\]
\end{lemma}

As in the proof of Proposition \ref{enmq_fixpoint}, denote
\begin{align*}
    \mathcal A_1^0&=\mathbb C,\quad \mathcal{A}_1^k=\mbox{span}\{s_{\mu_1}s_{\mu_2}^*,\ |\mu_1|=|\mu_2|=k,\ \mu_i\in\Lambda_n\},\quad k\in\mathbb N,
    \\
    \mathcal A_2^0&=\mathbb C,\quad \mathcal{A}_2^k=\mbox{span}\{t_{\nu_1}t_{\nu_2}^*,\ |\nu_1|=|\nu_2|=k,\ \nu_i\in\Lambda_m\},\quad k\in\mathbb N.
\end{align*}
Recall also that $\mathcal{A}_1^k\simeq M_{n^k}(\mathbb{C})$ and $\mathcal{A}_2^k\simeq M_{m^k}(\mathbb{C})$, see \cite{cun}.

Put $\mathcal{A}_q^{0}:=\mathbb{C}$,
\[
\mathcal{A}_q^k:=\sum_{k_1+k_2=k} \mathcal{A}_1^{k_1}\cdot\mathcal{A}_2^{k_2},
\]
and set
\[
\mathcal{A}_q=\overline{\bigcup_{k\in\mathbb{Z}_{+}} \mathcal{A}_q^k}.
\]
By Lemma \ref{lem_comm_rel}, for any
$x\in\mathcal{A}_1^k$ and $y\in\mathcal{A }_2^l$ one has $xy=yx$.
Thus $\mathcal{A}_q$ is an AF-subalgebra in $\mathcal{O}_n \otimes_q \mathcal{O}_m$ and $\mathcal{A}_{q_1}\simeq\mathcal{A}_{q_2}$ for any $q_1,q_2\in\mathbb{C}$, $|q_1|=|q_2|=1$.

To prove pure infiniteness of $\mathcal O_n\otimes_q\mathcal O_m$ we essentially follow Chapter V.4 of \cite{dav}.

Denote by ${Fin}_q^k$ the span of monomials $s_{\mu_1}s_{\mu_2}^* t_{\nu_1}t_{\nu_2}^*$ such that
\[
\max\{|\mu_1|,|\mu_2|\}+\max\{|\nu_1|,|\nu_2|\}\le k.
\]

\begin{proposition}
For any $k\in\mathbb{N}$ there exists an isometry $w_k\in\mathcal{O}_n\otimes_q \mathcal O_m$ such that
\[
E_{\alpha}(x)=w_k^* x w_k,\quad \mbox{for any}\ x\in Fin_q^k,
\]
and $w_k^* y w_k=y$ for any $y\in\mathcal{A}_q^k$.
\end{proposition}

\begin{proof}
Let $s_{\gamma}=s_1^{2k}s_2$ and $t_{\gamma}=t_1^{2k}t_2$. Consider the isometries
\[
w_{k,1}=\sum_{|\delta|=k,\delta\in\Lambda_n} s_{\delta}s_{\gamma}s_{\delta}^*,
\]
and
\[
w_{k,2}=\sum_{|\lambda|=k,\lambda\in\Lambda_m} t_{\lambda}t_{\gamma}t_{\lambda}^*.
\]
Then, see Lemma V.4.5 of \cite{dav},
\[
w_{k,1}^* s_{\mu_1}s_{\mu_2}^* w_{k,1}=0,\quad \mbox{if}\ |\mu_1|\ne |\mu_2|,\ |\mu_i|\le k, \mu_i\in\Lambda_n,
\]
and
\[
w_{k,1}^* s_{\mu_1}s_{\mu_2}^* w_{k,1}=s_{\mu_1}s_{\mu_2}^*,\quad \mbox{if}\ |\mu_1|=|\mu_2|,\ |\mu_i|\le k, \mu_i\in\Lambda_n.
\]
Analogously,
\[
w_{k,2}^* t_{\nu_1}t_{\nu_2}^* w_{k,2}=0,\quad \mbox{if}\ |\nu_1|\ne |\nu_2|,\ |\nu_i|\le k, \nu_i\in\Lambda_m,
\]
and
\[
w_{k,2}^* t_{\nu_1}t_{\nu_2}^* w_{k,2}=t_{\nu_1}t_{\nu_2}^*,\quad \mbox{if}\ |\nu_1|=|\nu_2|,\ |\nu_i|\le k, \nu_i\in\Lambda_m.
\]
By Lemma \ref{lem_comm_rel} we get
\begin{align*}
w_{k,1}t_{\nu_1}t_{\nu_2}^*&=\overline{q}^{(|\nu_1|-|\nu_2|)(2k+1)}t_{\nu_1}t_{\nu_2}^* w_{k,1},
\\
w_{k,1}^*t_{\nu_1}t_{\nu_2}^*&=q^{(|\nu_1|-|\nu_2|)(2k+1)}t_{\nu_1}t_{\nu_2}^* w_{k,1}^*,
\\
w_{k,2}s_{\mu_1}s_{\mu_2}^*&=q^{(|\mu_1|-|\mu_2|)(2k+1)}s_{\mu_1}s_{\mu_2}^* w_{k,2},
\\
w_{k,2}^*s_{\mu_1}s_{\mu_2}^*&=\overline{q}^{(|\mu_1|-|\mu_2|)(2k+1)}s_{\mu_1}s_{\mu_2}^* w_{k,2}^*.
\end{align*}
Then
\[
w_{k,2}w_{k,1}=q^{(2k+1)^2}w_{k,1}w_{k,2},\quad w_{k,2}^*w_{k,1}=\overline{q}^{(2k+1)^2}w_{k,1}w_{k,2}^*.
\]
Let $w_k=w_{k,2}w_{k,1}$. Evidently $w_k$ is an isometry. Then for any $|\mu_i|\le k$ and $|\nu_i|\le k$ one has
\[
w_k^* s_{\mu_1}s_{\mu_2}^* t_{\nu_1}t_{\nu_2}^* w_k=q^{\left((|v_1|-|v_2|)-(|\mu_1|-|\mu_2|)\right)(2k+1)}w_{k,1}^* s_{\mu_1}s_{\mu_2}^*w_{k,1}w_{k,2}^* t_{\nu_1}t_{\nu_2}^* w_{k,2},
\]
implying that for any  $|\mu_i| \le k$, and $|\nu_i|\le k$,
\[
w_k^* s_{\mu_1}s_{\mu_2}^* t_{\nu_1}t_{\nu_2}^* w_k=0,\quad \mbox{if}\ |\mu_1|\ne |\mu_2|\ \mbox{or}\ |\nu_1|\ne |\nu_2|,
\]
and
\[
w_k^* s_{\mu_1}s_{\mu_2}^*t_{\nu_1}t_{\nu_2}^* w_k=s_{\mu_1}s_{\mu_2}^*t_{\nu_1}t_{\nu_2}^*,\quad \mbox{if}\ |\mu_1|= |\mu_2|\ \mbox{and}\ |\nu_1|= |\nu_2|.
\]
Hence, for any $x\in Fin_q^k$ one has $w_k^* x w_k=E_{\alpha} (x)$ and
$w_k^* y w_k=y$ for $y\in\mathcal A_q^k$.
\end{proof}

\begin{remark}
Since $\mathcal{A}_q^k$ is finite-dimensional, it is a direct sum of full matrix algebras, where matrix units are represented by  $s_{\mu_1}t_{\nu_1}t_{\nu_2}^* s_{\mu_2}^*$, $|\mu_1|=|\mu_2|$, $|\nu_1|=|\nu_2|$ and $|\mu_1|+|\nu_1|=k$. In particular, any minimal projection in $\mathcal{A}_q^k$ is unitary equivalent in $\mathcal{A}_q^k$ to a ``matrix-unit projection'' having form $s_{\mu_1}t_{\nu_1}t_{\nu_1}^* s_{\mu_1}^*$ with $|\mu_1|+|\nu_1|=k$. So any minimal projection in $\mathcal{A}_q^k$ has the form $u\,u^*$ for some isometry $u\in\mathcal{E}_{n,m}^q$.
\end{remark}

The following statement is  the main result of this Subsection.

\begin{theorem}
For any non-zero $x\in \mathcal{O}_n \otimes_q \mathcal{O}_m$ with $|q|=1$, there  exist $a,b\in \mathcal{O}_n \otimes_q \mathcal{O}_m$ such that $axb=\mathbf1$.
\end{theorem}

\begin{proof}
The proof repeats the arguments of the proof of Theorem V.4.6  in \cite{dav}. We present it here for the reader's convenience.

Let $\mathcal O_n^q\otimes\mathcal O_m^q\ni x\ne 0$. Then $x^* x>0$ and $E_{\alpha} (x^*x)>0$. After normalisation of  $x$ we can suppose that $\|E_{\alpha}(x^*x)\|=1$. Find $k\in\mathbb{N}$ and $y=y^*\in Fin_q^k$ such that $\|x^*x-y\|<\frac{1}{4}$. Since $E_{\alpha}$ is a contraction, one has
\[
\|E_{\alpha} (x^*x)-E_{\alpha}(y)\|<\frac{1}{4}\quad \mbox{and}\quad \|E_{\alpha}(y)\|>\frac{3}{4}.
\]
Further, $w_k^* y w_k=E_{\alpha}(y)$. Since $E_{\alpha}(y)=E_{\alpha}(y)^*\in\mathcal{A}_q^k$, by the spectral theorem for a self-adjoint operator on a finite-dimensional Hilbert space, there exists a minimal projection $p\in\mathcal{A}_q^k$, such that
\[
pE_{\alpha}(y)=E_{\alpha}(y)p=\|E_{\alpha}(y)\|\cdot p.
\]

As noted above, $p= u\,u^*$ for an isometry $u\in\mathcal O_n\otimes_q\mathcal O_m$. Put
\[
z=\|E_{\alpha}(y)\|^{-\frac{1}{2}} u^* p w_k^*.
\]
Then $\|z\|<\frac{2}{\sqrt{3}}$, and
\begin{align*}
zyz^* &= \|E_{\alpha}(y)\|^{-1} u^* p w_k^* y w_k p u = \|E_{\alpha}(y)\|^{-1} u^* p E_{\alpha}(y)  p u
\\
&=\|E_{\alpha}(y)\|^{-1} \|E_{\alpha}(y)\|u^* p  u=u^* u u^* u=\mathbf1.
\end{align*}
Then
\begin{align*}
\|\mathbf1-z x^* x z^*\|&= \| z y z^* - z x^* x z^*\| \le \|z\|^2 \cdot \|y- x^* x\|<\frac{4}{3}\cdot\frac{1}{4}=\frac{1}{3}.
\end{align*}
Hence $z x^* x z^*$ is invertible in $\mathcal O_n\otimes_q\mathcal O_m$. Let $c\in\mathcal O_n\otimes_q\mathcal O_m$ satisfies ${c z x^* x z^*=\mathbf1}$, then for $a= c z x^*$ and $b= z^*$ one has $axb=\mathbf1$.
\end{proof}

The following corollary is immediate.

\begin{theorem}\label{Onmqsimple}
The $C^*$-algebra $\mathcal{O}_n \otimes_q \mathcal{O}_m$ is nuclear, simple and purely infinite.
\end{theorem}

Given $q = e^{2 \pi i \varphi_0}$, consider
\begin{equation}\label{theta_q1}
\Theta_q = \left( \begin{array}{cc}
    0 & \frac{\varphi_0}{2}  \\
    -\frac{\varphi_0}{2} & 0
\end{array} \right),
\end{equation}
and construct the Rieffel deformation $(\mathcal{O}_n\otimes  \mathcal{O}_m)_{\Theta_q}$.
\begin{corollary}\label{univ_thet_cun}
The following isomorphism holds:
 \[
 \mathcal{O}_n \otimes_q \mathcal{O}_m \simeq (\mathcal{O}_n\otimes  \mathcal{O}_m)_{\Theta_q}.
 \]
\end{corollary}

\begin{proof}
As in the proof of  Theorem  \ref{cuntoep_rieff}, the universal property of $\mathcal{O}_n\otimes_q\mathcal{O}_m$ implies that the correspondence
\[
s_j\mapsto s_j \otimes \mathbf1,\quad t_r\mapsto \mathbf1 \otimes t_r,\quad j=\overline{1,n},\ r=\overline{1,m},
\]
extends to a surjective homomorphism
$\Phi\colon\mathcal{O}_n\otimes_q\mathcal{O}_m\rightarrow (\mathcal{O}_n\otimes\mathcal{O}_m)_{\Theta_q}$.
Finally, the simplicity of $\mathcal{O}_n\otimes_q\mathcal{O}_m$ implies that $\Phi$ is an isomorphism.
\end{proof}

\begin{remark}\label{equivariant_rieff_isom}
The isomorphism established in Corollary \ref{univ_thet_cun} is equivariant with respect to the introduced above actions of $\mathbb{T}^2$ on $\mathcal{O} \otimes_q \mathcal O_m$ and $(\mathcal{O}_n\otimes  \mathcal{O}_m)_{\Theta_q}$ respectively.
\end{remark}

The simplicity of $\mathcal O_n\otimes_q\mathcal O_m$ implies that $\mathcal{M}_q\subset\mathcal E_{n,m}^q$ is the largest ideal.
\begin{corollary}\label{mq_unique}
 The ideal $\mathcal{M}_q\subset\mathcal E_{n,m}^q$ is the unique largest ideal.
\end{corollary}

\begin{proof}
 Let $\eta\colon\mathcal E_{n,m}^q\rightarrow \mathcal{O}_n\otimes_q\mathcal{O}_m$ be the quotient homomorphism.
 Suppose that $\mathcal{J}\subset\mathcal{E}_{n,m}^q$ is a two-sided $*$-ideal. Due to the simplicity of
 $\mathcal{O}_n\otimes_q\mathcal{O}_m$ we have that either $\eta(\mathcal J)=\{0\}$ and
 $\mathcal J\subset\mathcal M_q$, or $\eta(\mathcal J)= \mathcal{O}_n\otimes_q\mathcal{O}_m$. In the latter case,
 $\mathbf1+x\in \mathcal J$ for a certain $x\in\mathcal M_q$. For any $0<\varepsilon<1$, choose
 $N_{\varepsilon}\in \mathbb N$,
 such that for
 \[
 x_{\varepsilon}=\mspace{-9mu}
 \sum_{\substack{\varepsilon_1,\varepsilon_2 \in\{0,1\},\\
 \varepsilon_1+\varepsilon_2\ne 0}}
 \sum_{\substack{\mu_1,\mu_2\in\Lambda_n,\\ |\mu_j|\le N_{\varepsilon}}}
 \sum_{\substack{\nu_1,\nu_2\in\Lambda_m,\\ |\nu_j|\le N_{\varepsilon}}} \mspace{-9mu}
 \Psi_{\mu_1,\mu_2\nu_1\nu_2}^{(\varepsilon_1,\varepsilon_2)}s_{\mu_1}t_{\nu_1}
 (\mathbf 1-P)^{\varepsilon_1}(\mathbf 1-Q)^{\varepsilon_2}t_{\nu_2}^*s_{\mu_2^*}\in\mathcal M_q
 \]
one has $\|x-x_{\varepsilon}\|<\varepsilon$. Notice that for any $\mu\in\Lambda_n$, $\nu\in\Lambda_m$
with
$|\mu|,|\nu|> N_{\varepsilon}$ one has ${s_{\mu}^* t_{\nu}^* x_{\varepsilon}=0}$.

Fix $\mu\in\Lambda_n$ and $\nu\in\Lambda_m$, $|\mu|=|\nu|>N_{\varepsilon}$, then
\[
y_{\varepsilon}=s_{\mu}^*t_{\nu}^*(\mathbf 1- x)t_{\nu}s_{\mu}=\mathbf 1 - s_{\mu}^*
t_{\nu}^*(x- x_{\varepsilon})t_{\nu}s_{\mu}
\in\mathcal J.
\]
Thus $\|s_{\mu}^*t_{\nu}^*(x- x_{\varepsilon})t_{\nu}s_{\mu}\|<\varepsilon$ implies that $y_{\varepsilon}$ is
invertible,
so $\mathbf 1\in\mathcal J$.
\end{proof}

\subsection{The isomorphism $\mathcal{O}_n \otimes_q \mathcal{O}_m \simeq \mathcal{O}_n \otimes \mathcal{O}_m$}
In this section we prove the main result of Section 3. Namely, we show that for each $q$, $|q|=1$,
\[
\mathcal{O}_n \otimes_q \mathcal{O}_m \simeq \mathcal{O}_n \otimes \mathcal{O}_m.
\]

In \cite{Skandalis}, the authors have shown that
for every $C^*$-algebra $A$ with an action $\alpha$ of $\mathbb{R}$, there exists a KK-isomorphism $t_\alpha \in KK_{1}(A, A \rtimes_\alpha \mathbb{R})$. This $t_\alpha$ is a generalization of the Connes-Thom isomorphisms for K-theory. Below we will denote by $\circ : KK(A, B) \times KK(B, C) \rightarrow KK(A, C)$ the Kasparov product, and by $\boxtimes : KK(A, B) \times KK(C, D) \rightarrow KK(A \otimes C, B \otimes D)$ the exterior tensor product. Given a homomorphism $\phi : A \rightarrow B$, put $[\phi] \in KK(A, B)$ to be the induced KK-morphism. For more details see \cite{Black,Knudsen}.

We list some properties of $t_\alpha$ that will be used below.
\begin{enumerate}
    \item Inverse of $t_\alpha$ is given by $t_{\widehat{\alpha}}$, where $\widehat{\alpha}$ is the dual action.
    \item If $A = \mathbb{C}$ with the trivial action of $\mathbb{R}$, then the corresponding element
    \[
    t_1 \in KK_1(\mathbb{C}, C_0(\mathbb{R})) \simeq \mathbb{Z}
    \]
    is the generator of the group.
    \item Let $\phi: (A, \alpha) \rightarrow (B, \beta)$ be an equivariant homomorphism. Then the following diagram commutes in KK-theory
    \[
    \begin{tikzcd}
    A \arrow[r, "t_\alpha"] \arrow[d, "\phi"] & A \rtimes_\alpha \mathbb{R} \arrow[d, "\phi \rtimes \mathbb{R}"] \\
    B \arrow[r, "t_\beta"] & B \rtimes_\beta \mathbb{R}
    \end{tikzcd}
    \]
    \item Let $\beta$ be an action of $\mathbb{R}$ on $B$. For the action $\gamma = \id_A \otimes \beta$ on $A \otimes B$ we have
    \[
    t_\gamma = \mathbf1_A \boxtimes\, t_\beta.
    \]
\end{enumerate}

We will need the classification result by Kirchberg and Philips:

\begin{theorem}[\cite{Kirchberg}, Corollary 4.2.2]
Let $A$ and $B$ be separable nuclear unital purely infinite simple $C^*$-algebras, and suppose that there exists an invertible element $\eta \in KK(A,B)$, such that $[\iota_A] \circ \eta = [\iota_B]$, where $\iota_A : \mathbb{C} \rightarrow A$ is defined by  $\iota_A(1) =\mathbf1_A$, and $\iota_B : \mathbb{C} \rightarrow B$ is defined by $\iota_B(1) =\mathbf1_B$. Then $A$ and $B$ are isomorphic.
\end{theorem}

\begin{theorem}\label{q_stab_res}
The $C^*$-algebras $\mathcal{O}_n \otimes_q \mathcal{O}_m$ and $\mathcal{O}_n \otimes \mathcal{O}_m$ are isomorphic for any $|q| = 1$.
\end{theorem}

\begin{proof}
Throughout the proof we will distinguish between the actions of $\mathbb{T}^2$ on $\mathcal{O}_n \otimes \mathcal{O}_m$ and on $\mathcal{O}_n \otimes_q \mathcal{O}_m$, denoting the latter by $\alpha^q$. Due to Theorem \ref{Onmqsimple}, the both algebras are separable nuclear unital simple and purely infinite.

Further,  Corollary \ref{univ_thet_cun}, Proposition \ref{prop4}, and Remark \ref{equivariant_rieff_isom} yield the isomorphism
\[
\Psi : (\mathcal{O}_n \otimes \mathcal{O}_m) \rtimes_\alpha \mathbb{R}^2 \rightarrow (\mathcal{O}_n \otimes_q \mathcal{O}_m) \rtimes_{\alpha^q} \mathbb{R}^2.
\]
Decompose the crossed products as follows:
\begin{align*}
(\mathcal{O}_n \otimes \mathcal{O}_m) \rtimes_\alpha \mathbb{R}^2 &\simeq (\mathcal{O}_n \otimes \mathcal{O}_m) \rtimes_{\alpha_1} \mathbb{R} \rtimes_{\alpha_2} \mathbb{R},
\\
(\mathcal{O}_n \otimes_q \mathcal{O}_m) \rtimes_{\alpha^q} \mathbb{R}^2 &\simeq (\mathcal{O}_n \otimes_q \mathcal{O}_m) \rtimes_{\alpha_1^q} \mathbb{R} \rtimes_{\alpha_2^q} \mathbb{R}.
\end{align*}
Define
\begin{align*}
t_\alpha &= t_{\alpha_1} \circ (\mathbf1_{C_0(\mathbb{R})} \boxtimes t_{\alpha_2}) \in KK(\mathcal{O}_n \otimes \mathcal{O}_m,(\mathcal{O}_n \otimes \mathcal{O}_m) \rtimes_\alpha \mathbb{R}^2),
\\
t_{\alpha^q} &= t_{\alpha_1^q} \circ  (\mathbf1_{C_0(\mathbb{R})} \boxtimes t_{\alpha_2^q}) \in KK(\mathcal{O}_n \otimes_q \mathcal{O}_m,(\mathcal{O}_n \otimes_q \mathcal{O}_m) \rtimes_{\alpha^q} \mathbb{R}^2),
\end{align*}
Then
\[ \eta = t_{\alpha^q} \circ [\Psi] \circ t_{\alpha}^{-1} \in KK(\mathcal{O}_n \otimes_q \mathcal{O}_m, \mathcal{O}_n \otimes \mathcal{O}_m) \]
is a $KK$-isomorphism. The property $[\iota_{\mathcal{O}_n \otimes_q \mathcal{O}_m}] \circ \eta = [\iota_{\mathcal{O}_n \otimes \mathcal{O}_m}]$ follows from the commutativity of the following diagram
\[
\begin{tikzcd}[column sep = 0.5cm]
\mathbb{C} \arrow[rr, "t_1 \circ (\mathbf1_{C_0(\mathbb{R})} \boxtimes\, t_1)"] \arrow[d, "\iota_{\mathcal{O}_n \otimes_q \mathcal{O}_m}"]
& &
C_0(\mathbb{R}^2) \arrow[rr,"(\mathbf1_{C_0(\mathbb{R})} \boxtimes\, t_1)^{-1} \circ\, t_1^{-1}"] \arrow[dl,"\iota_{\mathcal{O}_n \otimes_q \mathcal{O}_m} \rtimes \mathbb{R}^2"'] \arrow[dr,"\iota_{\mathcal{O}_n \otimes \mathcal{O}_m} \rtimes \mathbb{R}^2"]
& &
\mathbb{C} \arrow[d, "\iota_{\mathcal{O}_n \otimes \mathcal{O}_m}" left] \\
\mathcal{O}_n \otimes_q \mathcal{O}_m \arrow[r,"t_{\alpha^q}"] &
(\mathcal{O}_n \otimes_q \mathcal{O}_m) \rtimes_{\alpha^q} \mathbb{R}^2 \arrow[rr,"\Psi"] & &
(\mathcal{O}_n \otimes \mathcal{O}_m) \rtimes_{\alpha} \mathbb{R}^2 \arrow[r,"t_{\alpha}^{-1}"] &
\mathcal{O}_n \otimes \mathcal{O}_m
\end{tikzcd}
\]
\end{proof}
\begin{remark}
After our paper was submitted, we were informed by Prof. M. Weber that in unpublished part of his  PhD thesis he studied a multi-parameter twisted tensor product of Cuntz algebras and obtained independently the analog of Theorem \ref{q_stab_res}.
\end{remark}

\subsection{Computation of $\mathsf{Ext}$ for $\mathcal{E}_{n,m}^q$}
Here we show that $\mathsf{Ext}(\mathcal{O}_n \otimes_q \mathcal{O}_m, \mathcal{M}_q)=0$ if $\gcd(n-1,m-1)=1$. We use the isomorphism $\mathcal O_n\otimes_q \mathcal O_m\simeq\mathcal O_n\otimes\mathcal O_m$, ${|q|=1}$.

Recall the notion of UCT property for $KK$-theory, see \cite{Black}.

\begin{definition}\label{uct}
Suppose $A$ and $B$ are separable nuclear $C^*$-algebras. We say that a pair $(A, B)$ satisfies the Universal Coefficient Theorem (UCT), if the following sequences are exact, $j\in\mathbb Z_2$,
\begin{equation*}
0 \rightarrow\bigoplus_{i\in\mathbb Z_2} Ext^1_\mathbb{Z}(K_{i+1}(A), K_{i+j}(B)) \rightarrow KK_{j}(A, B) \rightarrow \bigoplus_{i \in \mathbb Z_2} Hom(K_{i}(A), K_{i+j}(B)) \rightarrow 0. 
\end{equation*}
We say that $A$ satisfies UCT if $(A, B)$ satisfies UCT for every $B$.
\end{definition}

It is known that $\mathcal{O}_n \otimes_q \mathcal{O}_m \simeq \mathcal{O}_n \otimes \mathcal{O}_m$ satisfies UCT.
The following statement is an easy consequence of the Kunneth formula.

\begin{theorem}
Let $d = \gcd(n - 1, m - 1)$. Then
\[ 
K_0(\mathcal{O}_n \otimes_q \mathcal{O}_m) \simeq \mathbb{Z} / d\mathbb{Z}, \quad  K_1(\mathcal{O}_n \otimes_q \mathcal{O}_m) \simeq \mathbb{Z} / d\mathbb{Z}, 
\]
\end{theorem}
\begin{proof}
The Kunneth formula for K-theory, see \cite{Black} Theorem 23.1.3,  gives the following short exact sequences, $j\in\mathbb Z_2$,
\[
0 \rightarrow \bigoplus_{i\in\mathbb Z_2}  K_i(\mathcal{O}_n) \otimes_\mathbb{Z} K_{i + j}(\mathcal{O}_m) \rightarrow K_j(\mathcal{O}_n \otimes \mathcal{O}_m) \rightarrow \bigoplus_{i\in\mathbb Z_2} Tor_1^\mathbb{Z}(K_{i}(\mathcal{O}_n), K_{i + j + 1}(\mathcal{O}_m)) \rightarrow 0
\]
It is a well known fact in homological algebra, see \cite{McLane}, that for an abelian group $A$
\[
Tor_1^\mathbb{Z}(A, \mathbb{Z}/d\mathbb{Z}) \simeq Ann_A(d) = \{a \in A  \mid  da = 0 \}.
\]
In particular,
\[
Tor_1^\mathbb{Z}(\mathbb{Z}/n\mathbb{Z}, \mathbb{Z}/m\mathbb{Z}) \simeq \mathbb{Z}/\gcd(n,m)\mathbb{Z}.
\]
Recall that, see \cite{Cuntz_Ktheory},
\[
K_0(\mathcal{O}_n) = \mathbb{Z} / (n - 1) \mathbb{Z}, \quad K_1(\mathcal{O}_n) = 0.
\]
Hence, for $\mathcal{O}_n \otimes \mathcal{O}_m$, one has the following short exact sequences:
\begin{gather*}
0 \rightarrow \mathbb{Z}/(n-1)\mathbb{Z} \otimes_\mathbb{Z} \mathbb{Z}/(m-1)\mathbb{Z} \rightarrow K_0(\mathcal{O}_n \otimes \mathcal{O}_m) \rightarrow 0 \rightarrow 0,
\\
0 \rightarrow 0 \rightarrow K_1(\mathcal{O}_n \otimes \mathcal{O}_m) \rightarrow \mathbb{Z} / d \mathbb{Z} \rightarrow 0.\qedhere
\end{gather*}
\end{proof}

Next step is to compute the K-theory of $\mathcal{M}_q$.

\begin{theorem}
Let $d = \gcd(n - 1, m - 1)$. Then
\[
K_0(\mathcal{M}_q) \simeq \mathbb{Z} / d\mathbb{Z} \oplus \mathbb{Z}, \text{ }  K_1(\mathcal{M}_q) \simeq 0.
\]
\end{theorem}

\begin{proof}
By Theorem \ref{cuntoep_rieff}, Proposition \ref{Rieff_K_theory}, and \cite{Cuntz_Ktheory}, Proposition 3.9,
\begin{align*}
K_0(\mathcal{E}_{n,m}^q) &= K_0((\mathcal{O}_n^{(0)}\otimes\mathcal{O}_m^{(0)})_{\Theta_q}) = K_0(\mathcal{O}_n^0 \otimes \mathcal{O}_m^0) = \mathbb{Z},
\\
K_1(\mathcal{E}_{n,m}^q) &= K_1((\mathcal{O}_n^{(0)}\otimes\mathcal{O}_m^{(0)})_{\Theta_q}) = K_1(\mathcal{O}_n^0 \otimes \mathcal{O}_m^0) = 0.
\end{align*}
Applying the 6-term exact sequence for
\[ 0
\rightarrow \mathbb{K} \rightarrow \mathcal{M}_q \rightarrow \mathcal O_n\otimes\mathbb K \oplus \mathcal O_m\otimes\mathbb K \rightarrow 0 ,
\]
we get
\[
\begin{tikzcd}
\mathbb{Z} \arrow[r] & K_0(\mathcal{M}_q) \arrow[r] & \mathbb{Z}/(n-1)\mathbb{Z} \oplus \mathbb{Z}/(m-1)\mathbb{Z} \arrow[d] \\
0 \arrow[u] & \arrow[l] K_1(\mathcal{M}_q) & \arrow[l] 0
\end{tikzcd}
\]
Then  $K_1(\mathcal{M}_q) = 0$,
and elementary properties of finitely generated abelian groups imply that
\[ K_0(\mathcal{M}_q) = \mathbb{Z} \oplus \mathsf{Tors},\] where $\mathsf{Tors}$ is a direct sum of finite cyclic groups.

Further, the following exact sequence
\[
0 \longrightarrow \mathcal{M}_q \longrightarrow \mathcal{E}_{n,m}^q \rightarrow \mathcal{O}_n \otimes_q \mathcal{O}_m\longrightarrow 0
\]
gives
\[
\begin{tikzcd}
K_0(\mathcal{M}_q) \arrow[r, "p"] & \mathbb{Z} \arrow[r] & \mathbb{Z} / d \mathbb{Z} \arrow[d] \\
\mathbb{Z} / d \mathbb{Z} \arrow[u, "i"] & \arrow[l] 0 & \arrow[l] 0
\end{tikzcd}
\]
The map $p : K_0(\mathcal{M}_q) \simeq \mathbb{Z} \oplus \mathsf{Tors} \rightarrow \mathbb{Z}$ has form $p=(p_1,p_2)$, where
\[
p_1\colon\mathbb Z\rightarrow\mathbb Z,\quad p_2\colon\mathsf{Tors}
\rightarrow\mathbb Z.
\]
Evidently, $p_2=0$, and $p\ne 0$ implies that $\ker p_1=\{0\}$.
Thus,
\[
\ker p = \mathsf{Tors} = \mathrm{Im}(i) \simeq \mathbb{Z} / d\mathbb{Z}.
\qedhere
\]
\end{proof}

\begin{theorem}\label{Ext_comp}
Let $d=\gcd(n - 1, m - 1)=1$. Then $\mathsf{Ext}(\mathcal{O}_n \otimes_q \mathcal{O}_m, \mathcal{M}_q)=0$.
\end{theorem}

\begin{proof}
Recall that for nuclear $C^*$-algebras $\mathsf{Ext}(A, B) \simeq KK_1(A, B)$. 

We use  the sequence from Definition \ref{uct} for $j=1$, $A = \mathcal{O}_n \otimes_q \mathcal{O}_m$ and $B = \mathcal{M}_q$:
\[ 
0 \rightarrow \bigoplus_{i \in \mathbb{Z}_2} Ext_\mathbb{Z}^1(K_i(A), K_i(B)) \rightarrow KK_1(A, B) \rightarrow \bigoplus_{i \in \mathbb{Z}_2} {Hom}(K_i(A), K_{i + 1}(B)).
\]
Since $K_0(A)=K_1(A)=\mathbb Z/d\mathbb Z$ and $K_0(B)=\mathbb Z\oplus\mathbb Z/d\mathbb Z$, $K_1(B)=0$, one has
\[ 
{Hom}(K_0(A), K_1(B)) = 0, \quad {Hom}(K_1(A), K_0(B)) = \mathbb{Z} / d \mathbb{Z},
\]
and, see \cite{McLane},
\[ 
Ext_\mathbb{Z}^1(K_0(A), K_0(B)) = \mathbb{Z} / d \mathbb{Z} \oplus \mathbb{Z} / d \mathbb{Z},\quad \ Ext_\mathbb{Z}^1(K_1(A), K_1(B)) = 0.
\] 
Hence the following sequence is exact
\[ 0
\rightarrow \mathbb{Z} / d \mathbb{Z} \oplus \mathbb{Z} / d \mathbb{Z} \rightarrow KK_1(\mathcal{O}_n \otimes_q \mathcal{O}_m, \mathcal{M}_q) \rightarrow \mathbb{Z} / d \mathbb{Z} \rightarrow 0. 
\]
\end{proof}

By Theorem \ref{Ext_comp}, for the case of $\gcd(n-1, m-1)=1$ one can immediately deduce that extension classes of
\[ 0
\rightarrow \mathcal{M}_q \rightarrow \mathcal{E}_{n,m}^q \rightarrow \mathcal{O}_n \otimes_q \mathcal{O}_m \rightarrow 0,
\]
and
\[ 0
\rightarrow \mathcal{M}_1 \rightarrow \mathcal{E}_{n,m}^1 \rightarrow \mathcal{O}_n \otimes \mathcal{O}_m \rightarrow 0,
\]
coincide in $\mathsf{Ext}(\mathcal{O}_n \otimes \mathcal{O}_m, \mathcal{M}_1)$ and are trivial. These extensions are essential, however in general case one does not have an immediate generalization of Proposition \ref{Voiculescu}. Thus the study of the problem whether $\mathcal{E}_{n,m}^q \simeq \mathcal{E}_{n,m}^1$ would require further investigations, see \cite{E_theory,Eilers}.
\ \\

\textbf{Acknowledgements}
The work on the paper was initiated during the visit  of V. Ostrovskyi, D. Proskurin and R. Yakymiv to Chalmers University of Technology. We appreciate the working atmosphere and stimulating discussions with Prof. Lyudmila Turowska and Prof. Magnus Goffeng. We are grateful to Prof. Eugene Lytvynov for fruitful discussions on generalized statistics and his kind hospitality during the visit of D. Proskurin to Swansea University.
We also indebted to Prof. K. Iusenko for helpful comments and remarks.



\end{document}